\documentclass[12pt]{article}
\usepackage{amsmath}
\usepackage{graphicx}
\usepackage{enumerate}
\usepackage{natbib}
\usepackage{url} 

\usepackage{hyperref,amsthm,amsmath,amssymb,natbib,enumerate,amsbsy,amsfonts, xcolor, graphics,graphicx, comment, algpseudocode, algorithm2e, bbm, booktabs, multirow,url, makecell, float, mathtools, pgfplots}

\usepackage{arxiv}
\usepackage{natbib}
\usepackage[nottoc]{tocbibind}
\usepackage{amsmath} 
\usepackage{comment}
\usepackage[utf8]{inputenc} 
\usepackage[T1]{fontenc}    
\usepackage{hyperref}       
\usepackage{url}            
\usepackage{booktabs}       
\usepackage{amsfonts}       
\usepackage{nicefrac}       

\usepackage{microtype}      
\usepackage{lipsum}
\usepackage{graphicx}
\usepackage{xspace}
\usepackage{enumitem}
\usepackage{dsfont}
\usepackage{graphicx}   
\usepackage{wrapfig}
\usepackage{subcaption}  
\usepackage{mathrsfs}
\usepackage{float} 
\usepackage{natbib}
\usepackage{babel}
\usepackage{ amssymb, mathtools}
\usepackage{mismath}
\usepackage{tikz}

\graphicspath{ {./images/} }

\usepackage{romannum}
\newcommand{\iid}{i.i.d.\@\xspace}
\newcommand{\identite}{\mathrm{I}}

\newtheorem{proposition}{Proposition}
\newtheorem{theorem}{Theorem}
\newtheorem{definition}{Definition}
\newtheorem{lemma}{Lemma}

\newtheorem{remark}{Remark}
\newtheorem{assumption}{Assumption}

\renewcommand{\theequation}{
\arabic{equation}%
}

\pgfplotsset{compat=1.18}

\usepackage{xcolor}
\hypersetup{
    colorlinks,
    linkcolor={blue},
    citecolor={blue},
    urlcolor={black}
}

\def\spacingset#1{\renewcommand{\baselinestretch}%
{#1}\small\normalsize} \spacingset{1}

\title{Robust Regression under Adversarial Contamination:\\
Theory and Algorithms for the Welsch Estimator}

\author{
  Ilyes Hammouda \\
  CREST, ENSAE \\
  \texttt{ilyes@ensae.fr}
  \And
  Mohamed Ndaoud \\
  ESSEC Business School \\
  \texttt{mohamed.ndaoud@essec.edu}
  \And
  Abd\!-Krim Seghouane \\
  The University of Melbourne \\
  \texttt{aks@unimelb.edu.au}
}

\date{} 

\begin{document}
\maketitle

\begin{abstract}

Convex and penalized robust regression methods often suffer from a persistent bias induced by large outliers, limiting their effectiveness in adversarial or heavy-tailed settings. In this work, we study a smooth redescending non-convex M-estimator, specifically the Welsch estimator, and show that it can eliminate this bias whenever it is statistically identifiable. We focus on high-dimensional linear regression under adversarial contamination, where a fraction of samples may be corrupted by an adversary with full knowledge of the data and underlying model.

A central technical contribution of this paper is a practical algorithm that provably finds a statistically valid solution to this non-convex problem. We show that the Welsch objective remains locally convex within a well-characterized basin of attraction, and our algorithm is guaranteed to converge into this region and recover the desired estimator. We establish three main guarantees:
(a) non-asymptotic minimax-optimal deviation bounds under contamination,
(b) improved unbiasedness in the presence of large outliers, and
(c) asymptotic normality, yielding statistical efficiency as the sample size grows. Finally, we support our theoretical findings with comprehensive experiments on synthetic and real datasets, demonstrating the estimator’s superior robustness, efficiency, and effectiveness in mitigating outlier-induced bias relative to state-of-the-art robust regression methods.
\end{abstract}

\noindent%
\pagenumbering{arabic}
\section{Introduction}
In the field of statistical modeling and machine learning, robustness to outliers is a critical attribute for practical applications. In general, outliers refer to one or more observations that are markedly different from the bulk of the data and a routine data set may contain 1\% to 10\% (or more) outliers \citep{rousseeuw1986robust}. Linear regression, as one of the cornerstone methodologies, remains ubiquitous in fields, ranging from economics—where it is applied to understand market trends and forecast outcomes \citep{hellwig2014linear}—to sociology, where meta-analysis is used to explore power dynamics and societal heterogeneity \citep{tong2022meta}. Despite their growing importance, traditional methods for leveraging big databases for predictions and identifying causal relationships often lack the robustness required to handle the complexity and variability when observations are corrupted by outliers. The family of Ordinary Least Squares (OLS) estimators is well-suited for dense linear regression problems, offering optimal performance under ideal conditions. However, these estimators are unreliable and exhibit significant sensitivity to anomalous data or heavy-tailed error distributions.
As demonstrated by \citet{SeheultRobust1989}, the presence of even a single outlier with arbitrarily large values can cause the mean squared error of OLS estimators to escalate dramatically, severely undermining their reliability and robustness. Some studies, as in \cite{weisberg1982residuals}, propose the use of data management techniques to address outliers. These methods involve applying tests and diagnostics to identify and remove contaminated observations. They are simple and effective when there is only a single outlier. However, these methods can fail when there are multiple outliers. The main challenge in this case is to counter the masking and swamping effects \citep{hadi1993procedures}.
Additionally, they are largely based on heuristics, with a lack of established theoretical results.


\begin{figure}[H]
    \centering
    \begin{subfigure}[b]{0.48\textwidth}
        \centering
       \scalebox{0.7}{%
    \begin{tikzpicture}
    \begin{axis}[
        width=11cm, height=10cm,
        xlabel={Outlier proportion},
        ylabel={$\left\| \mathbf{E}(\hat{\beta}) - \beta^*\right\|$},
        ymin=0, ymax=1,
        xtick={0,0.02,...,1},
        grid=major,
        grid style={gray!30},
        axis lines=left,
        legend style={at={(0.5,-0.15)}, anchor=north, legend columns=2} 
    ]
    \addplot[
        color=blue,
        thick,
        mark=o
    ] 
    table[
       col sep=semicolon,
        x=delta,
        y=Error
    ] {csv_files/bias_Huber_big_outliers.csv};
    \addlegendentry{Huber}

    \addplot[
        color=red,
        thick,
        mark=x
    ] 
    table[
       col sep=semicolon,
        x=delta,
        y=Error
    ] {csv_files/biais_alpha_div_big_outliers.csv};
    \addlegendentry{Welsch}

    \addplot[
        color=cyan,
        dashed,
    ] 
    table[
       col sep=comma,
        x=delta,
        y=Error
    ] {csv_files/biais_quantile_reg_q_1.csv};
    \addlegendentry{Quantile Regression: $\tau=0.1$}

    \addplot[
        color=brown,
        dashed,
    ] 
    table[
       col sep=comma,
        x=delta,
        y=Error
    ] {csv_files/biais_quantile_regression_q_5.csv};
    \addlegendentry{Quantile Regression: $\tau=0.5$}

    \addplot[
        color=violet,
        dashed,
    ] 
    table[
       col sep=comma,
        x=delta,
        y=Error
    ] {csv_files/biais_quantile_regression_q_9.csv};
    \addlegendentry{Quantile Regression: $\tau=0.9$}

    \end{axis}
\end{tikzpicture}}
    \caption{\small Welsch compared to non-redescending}
    \label{fig::y}
    \end{subfigure}
    \begin{subfigure}[b]{0.48\textwidth}
        \centering
    \scalebox{0.7}{%
    \begin{tikzpicture}
    \begin{axis}[
        width=11cm, height=10cm,
        xlabel={Outlier proportion},
        ylabel={$\left\| \mathbf{E}(\hat{\beta}) - \beta^*\right\|$},
        ymin=0, ymax=0.015,
        xtick={0,0.02,...,1},
        xmin=0, xmax=0.1,
        grid=major,
        grid style={gray!30},
        axis lines=left,
        legend style={at={(0.5,-0.15)}, anchor=north, legend columns=2} 
    ]
    \addplot[
        color=blue,
        thick,
        mark=o
    ] 
    table[
       col sep=semicolon,
        x=delta,
        y=Error
    ] {csv_files/bias_tukey_high_outliers.csv};
    \addlegendentry{Tukey’s biweight}

    \addplot[
        color=red,
        thick,
        mark=x
    ] 
    table[
       col sep=semicolon,
        x=delta,
        y=Error
    ] {csv_files/biais_alpha_div_big_outliers.csv};
    \addlegendentry{Welsch}
        \addplot[
        color=cyan,
        dashed,
    ] 
    table[
       col sep=semicolon,
        x=delta,
        y=Error
    ] {csv_files/bias_Hampel_high_outliers.csv};
    \addlegendentry{Hampel’s three-part}
    \end{axis}
\end{tikzpicture}}
        \caption{ Welsch compared to hard redescending}
        \label{fig::b}
    \end{subfigure}
    \caption{\small Comparison of the Euclidean norm of the bias of the Welsch estimator against other robust estimators.}
    \label{fig::bias_figure}
\end{figure}

Thus, the study of $M$-estimators with inherent robust properties is essential for addressing challenges in statistical estimation \citep{rousseeuw1984least,HuberRobustLocation1964}, particularly in the presence of outliers. Prior works have demonstrated that the Huber estimator exhibits interesting robustness properties in handling outliers  \citep{Outlier-robust_D,sasai2020robust, minsker2024robust}, and asymptotically converges to a normal distribution under the presence of heavy-tailed noise \citep{Distributed_S}. It has also been used in various applications \citep{sun2020adaptive,liu2024robust}. Yet it is known that non-convex loss functions are more effective in dealing with multiple gross outliers with possibly high leverage values \citep{huber1981robust,shevlyakov2008redescending}. Hence in low-dimensional problems, a key motivation for using a non-convex loss is its enhanced statistical efficiency. Indeed, the Huber estimator suffers from an inherent bias, especially in the presence of large magnitude outliers, as highlighted in Figure \ref{fig::bias_figure}. This experiment was conducted using synthetically generated data in accordance with the linear regression model specified in \eqref{model}, incorporating heavy-tailed noise drawn from a Pareto distribution and varying levels of outlier contamination. The underlying true parameter vector $\beta^*$, was fixed throughout the experiment to provide a consistent ground truth for performance evaluation. At each iteration, the proportion of data points contaminated by large outliers was systematically increased to assess the robustness of the estimator under escalating adversarial conditions. For each specified contamination level, we independently generated $5000$ realizations of the design matrix $X$ and the response vector $Y$ according to \eqref{model}. In each realization, the estimator of $\beta^*$ was computed, and the empirical bias was subsequently estimated by averaging the deviations between the estimated and true parameter vectors across the $5000$ trials. This large-scale simulation framework allows for a statistically reliable assessment of estimator's bias under varying degrees of contamination, thereby highlighting the limitations of convex robust methods, such as Huber and quantile regression, in extreme settings.

In this paper, we investigate the minimum distance estimator obtained from an $\alpha$-divergence, commonly known as the Welsch estimator \citep{dennis1978techniques},  under a Gaussian nominal model. Methods based on $\alpha$-divergence, viewed as a generalization of classical divergence measures, have recently emerged as a versatile tool for enhancing robustness in Bayesian settings (\cite{hernandez2016black}, \cite{li2017dropout}). We establish theoretical guarantees showing that the Welsch estimator attains minimax optimality when the data are subject to heavy‐tailed noise and contaminated by outliers. To the best of our knowledge, these are the first non-asymptotic optimality results for the Welsch estimator.
Finally, we illustrate our theoretical findings through experiments on both synthetic and real‐world datasets, comparing the Welsch estimator against previously proposed hard redescending estimators (Tukey’s biweight estimator \citep{beaton1974fitting}, Hampel’s three‑part M‑estimators \citep{hampel1974influence}) as well as classical convex robust estimators such as Huber’s M-estimator \citep{Huber1973Robustregression}.  
\subsection{Related work}
\label{alpha_div_section}
Consider the classical linear regression model
\begin{equation}
    Y_i = X_i^\top \beta^* + \xi_i,
    \quad
    Y_i \in \mathbb{R},\; X_i \in \mathbb{R}^p,\; i = 1,\dots,n,
\end{equation}
where $\xi_i$ denotes an independent noise term.  Classical M‑estimators in robust regression \citep{Huber1973Robustregression} seek
\begin{equation}
    \widehat\beta_{\mathrm{M}}
    =
    \arg\min_{\beta\in\mathbb{R}^p}
    \sum_{i=1}^n \rho\bigl(Y_i - X_i^\top \beta\bigr),
\end{equation}
for a loss function $\rho$. The notion of redescending appeared first in \citep{princeton}. On the one hand, Huber’s estimator is based on a convex loss function and hence is considered non-redescending as it assigns nonzero weights to arbitrarily large residuals.
 On the other hand, redescending estiamtors (e.g.\ Welsch’s estimator \citep{dennis1978techniques}, Tukey’s biweight \citep{beaton1974fitting} and Hampel’s three‑part M‑estimators \citep{hampel1974influence}) employ influence functions that taper continuously to zero beyond a threshold, assigning almost zero weight to large residuals. This largely nullifies the impact of gross outliers, achieving a higher breakdown point (up to $0.5$) and superior bias control in the presence of extreme contamination \citep{hampel1975beyond,maronna2006robust,SeheultRobust1989,Highbreakdownpoint}.  From a practical perspective, \cite{she2011outlier} demonstrates that appropriately chosen non-convex penalty functions significantly improve outlier detection performance within penalized regression frameworks. These methods outperform convex penalization approaches, such as those based on the Huber loss, which fail to adequately address challenging issues like masking and swamping effects.

In this study we aim to analyze the inherent robustness properties of the Welsch estimator. We remind the reader of the corresponding loss function: 
\begin{equation}
\label{alpha_div_based_loss_function}
    l_{\tau}(x)=\frac{1}{\tau} \left( 1-\exp\left( \frac{-\tau x^2 }{2}\right)\right),\ \tau>0.
\end{equation}

\begin{figure}[!t]
    \centering 
    \scalebox{1}{%
    \begin{tikzpicture}
    \begin{axis}[
        width=12cm, 
        height=6cm, 
        xlabel={$x$}, 
        ylabel={$y$}, 
        xmin=-3, xmax=3, 
        ymin=0, ymax=2, 
        axis lines=middle, 
        legend style={at={(0,1)},anchor=north west},
        grid=major,
        grid style={ gray!30},
    ]
        \addplot[
            samples=200, 
            thick,
            dashed,
            red, 
        ] {(1 - exp(-0.5 * x^2))};
        \addlegendentry{Welsch} 
        \addplot[
            domain=-3:-1, 
            samples=100,
            thick,
            blue,
            no markers
        ] { -x - 1/2 };

        \addplot[
            domain=-1:1, 
            samples=100,
            thick,
            blue,
            no markers
        ] { x^2 / 2 };

        \addplot[
            domain=1:3, 
            samples=100,
            thick,
            blue,
            no markers
        ] { x - 1/2 };

        \addlegendentry{Huber} 


    \end{axis}
\end{tikzpicture}}
    	\caption{\label{functions}\small Comparison between Huber and Welsch loss functions.} 
\end{figure}

Figure~\ref{functions} shows the Welsch and Huber loss functions, illustrating the difference between non-redescending and redescending estimators: both behave quadratically near the origin but flatten beyond their respective cutoffs, ensuring that outliers fall within the plateau region and contribute negligibly to the total loss \citep{HuberRobustLocation1964,huber2011robust}.  However, Welsch’s non‑convex shape yields enhanced stability, particularly under heavy‑tailed or arbitrarily corrupted noise, while Huber’s convexity admits only partial mitigation of outlier influence.  We will demonstrate these practical advantages in Section~\ref{sec:simu}.

Despite their empirical success in applications ranging from computer vision to bioinformatics \citep{de2021review,tian2022recent,hippke2019wotan}, non‑convex M‑estimators have received comparatively little theoretical treatment, in contrast with convex M-estimators \citep{Martinminmax,yu2017robust,hampel1981change}, owing largely to the analytical challenges posed by non‑convex optimization.  Although the Welsch loss function has historically appeared as an ad-hoc choice, a recent work \citep{Iqbalalphadivbased} have established a formal connection between this loss and the $\alpha$-divergence \citep{renyi1961measures,csiszar1967information,ali1966general}, highly studied in the Bayesian statistics literature \citep{beran1977minimum,tamura1986minimum,simpson1987minimum,lindsay1994efficiency,hooker2014bayesian}. Recent works in Bayesian statistics have further investigated both theoretical and practical aspects of the $\alpha$-divergence and its related measures within the framework of Variational Inference \citep{li2016renyi,daudel2023alpha,daudel2021infinite,bui2016black,dieng2017variational,daudel2023monotonic,daudel2021mixture,rodriguez2022adversarial}. For example, \cite{li2016renyi} examines the application of $\alpha$-divergence in message-passing algorithms. This connection highlights the inherent robustness properties of the Welsch estimator. Notably, \cite{knoblauch2022optimization} demonstrates that Rényi's $\alpha$-divergence \citep{renyi1961measures} exhibits several theoretically appealing characteristics relative to the $\beta$ and $\gamma$ divergences, particularly with respect to maintaining prior robustness without inducing artificial shrinkage in marginal variances.\\
Building on these insights, we establish, to the best of our knowledge, the first deviation bound for the Welsch estimator. Our results show that it achieves minimax optimality with respect to adversarial contamination and yields lower estimation bias compared to contemporary state-of-the-art estimators. Furthermore, we propose a general two-stage analytical framework for studying non-convex loss functions. This framework may be extended to a broader class of non-convex robust losses.

\subsection{Contributions}
Section \ref{sec:theory} is devoted to the theoretical properties of the Welsch estimator. Our contribution is threefold:
\begin{itemize}[label=-]
\item In Section \ref{subsec:alpha}, we provide a comprehensive analysis of the Welsch loss landscape and discuss the inherent robustness properties of the objective function. Our key finding is that the objective function is strictly convex, with high probability, over a suitably chosen basin of attraction. This result emphasizes the practicality of this procedure, since most popular known robust estimators such as the Huber estimator or the LAD estimator fall in a convex set over which the Welsch loss function is convex. This, in turns, bypasses the non-convexity limitation of the latter. Our analysis may be of independent interest, as it provides a different perspective for rigorously analyzing non-convex estimation procedures, for instance other redescending procedures.

\item Next, we derive, in Section \ref{subsec:nonasymp}, a non-asymptotic result indicating that the convergence rate of the estimator is of the order $\left(\sqrt{\frac{p}{n}}+\frac{o}{n} \right)$ when the noise is sub-Gaussian, up to a logarithmic factor, and $\left(\sqrt{\frac{p}{n}}+\sqrt{\frac{o}{n}}\right)$, when the noise only has finite variance. As a consequence the Welsch estimator is minimax optimal. To the best of our knowledge, this is the first theoretical result establishing robustness optimality for this estimator. We derive our results within an adversarial‐contamination framework that entails greater analytical complexity than the traditional Huber contamination model.
Additionally, we establish a second non-asymptotic result showing that our method can enjoy bias-free deviations when the outliers are large enough enhancing the superiority of this method over Huber's estimator. 

\item Lastly, we prove, in Section \ref{subsec:asymp} that, asymptotically, the estimator converges to a normal distribution of variance $1$, under minimal Assumptions, showing its efficiency.
\end{itemize}

 Our theoretical findings are corroborated by numerical experiments in Section \ref{sec:simu} where we present the results of simulations performed on both synthetic and real-world data. All the proofs are deferred to the Supplementary Material.

\subsection{Adversarial contamination in linear regression}
In what follows, we assume that the data is generated according to the following mean shift linear regression model: 
\begin{equation}
    \mathbf{Y}=\mathbf{X} \beta^* + \theta + \xi,
    \label{model}
\end{equation}
Where
$\mathbf{Y} := \begin{pmatrix} Y_1, Y_2, \cdots, Y_n \end{pmatrix}^T \in \mathbf{R}^n$, 
$\mathbf{X}:=\left[X_1; X_2 ;\cdots;X_n \right]^T \in \mathbf{R}^{n \times p}$, 
$\beta^* := \begin{pmatrix} \beta_1, \beta_2, \cdots, \beta_p \end{pmatrix}^T \in \mathbf{R}^p$, 
$\theta := \begin{pmatrix} \theta_1 , \theta _2, \cdots, \theta _n \end{pmatrix}^T \in \mathbf{R}^n$ and 
$\xi := \begin{pmatrix} \xi_1 , \xi _2, \cdots, \xi _n \end{pmatrix}^T \in \mathbf{R}^n$.
We suppose that the $n$-data samples $(X_1,Y_1), \cdots,(X_n,Y_n)$ are independent and identically distributed (\iid).  These observations are also affected by a noise term \((\xi_i)_{i=1}^{n} \in \mathbf{R}\) whose inputs are assumed to be \iid, of unit variance and zero mean, and independent of the regression vectors \(X_i\). Finally, we consider that the measurements \((Y_i)_{i=1}^{n}\) are contaminated by an adversarial noise term, modeled in equation (\ref{model}) by the term \(\theta_i\). This term is generated by an adversary with access to the data $((Y_i, X_i, \xi_i)_{i=1}^{n})$, to the vector $\beta^*$, as well as to the joint distribution of all the variables in the model.  The adversarial contamination framework has recently gathered increased attention compared to the classical Huber contamination model \citep{HuberRobustLocation1964}, as it more accurately captures the nature of modern security-sensitive applications by explicitly modeling the malicious intent of attackers through training-data poisoning. \citep{dalvi2004adversarial,shafahi2018poison,adversial_attacks_survey_2024}.  

This paper focuses on the case of dense linear regression, where the dimension $n$ is assumed to be significantly larger than $p$, and the vector $\beta^*$ has no particular structure. Moreover, we assume that the noise terms $(\xi_i)_{i=1}^{n}$ follow a heavy-tailed distribution. We aim to study the statistical properties of the Welsch estimator, for a well chosen value of $\tau$,
\begin{equation}
\label{estimator}
    \hat{\beta}:= \arg\min_{\beta \in \mathbf{R}^p } \frac{1}{ n} 
    \sum_{i=1}^n
    \frac{1}{\tau} \left( 1-\exp\left( \frac{-\tau \left(Y_i-X_i^T \beta \right)^2 }{2}\right)\right).
\end{equation}

\subsection{Notations}
\label{notation}
Absolute constants independent of any problem parameters will be denoted by symbols such as $C$, as well as $c$. By convention, capital $C$ will represent ``a sufficiently large absolute constant'', whereas lowercase $c$ will indicate ``a sufficiently small absolute constant''. \\ 
 We note \( \identite_p \) the identity matrix of size $p \times p$. For vectors $u,v$ in $\mathbf{R}^p$, the $\ell_2$- norm is defined as follows: $\lVert u \lVert^2:= \sum_{i=1}^n u_i^2$ and the corresponding scalar product $\langle u , v\rangle = \sum_{i=1}^n u_iv_i$. In the case of a matrix  $A \in \mathbf{R}^{n \times p }$ we note its spectral norm as follows: $\lVert A \lVert_{\infty}:= \sup_{\|x\| \leq 1} \|Ax\|$. Furthermore, the Moore–Penrose inverse of $A$, denoted by $A^\dag$ is defined such that $A^\dag$ satisfies the following property: For all $y$ in the span of the columns of $A$, we have that \(\mathbf{A}\mathbf{A}^\dag y = y\). The singular values of $A$ are noted as follows $(s_i(A))_{i \in [1, n]}$, such that $s_1(A) \geq \dots \geq s_n(A) $. Finally $(\lambda_i(A))_{i \in [1, n]}$ are the eigenvalues of the symmetric matrix $A$, also ordered in descending order. Finally, $O$ will refer to the support of the vector $\theta$ defined in  \eqref{model}, ie the subset of $\{ 1,\cdots,n\}$ for which $\theta_i \neq 0$. For two matrices of the same dimension,  $A ,B\in \mathbf{R}^{n \times p }$, the Hadamard product, denoted by $A\otimes B$ is defined as the element-wise product of the two matrices. We use $o$ to refer to $|O|= \text{Card}(O)$. In this work, we will also use notations $\left(\xrightarrow{d} \right)$ to denote convergence in distribution, $\left(\xrightarrow{a.s.}\right)$ for almost sure convergence and $\left(\xrightarrow{\mathbf{P}} \right)$ for convergence in probability.

\section{On the theoretical properties of the Welsch  estimator}
\label{sec:theory}
In this section, we investigate the statistical properties of the estimator defined in \eqref{estimator}. We first state the Assumptions under which our results are valid.

\begin{assumption}
\label{assmp1}
The outliers vector $\theta$ is $o$-sparse i.e. $\sum_{i=1}^n \mathds{1}(\theta_i \neq 0)\leq o$.
\end{assumption}

\begin{assumption}
\label{assmp2}
For \( i = 1, \ldots, n \), we suppose that \( X_1, \ldots, X_n \) are centered, isotropic and $1$-sub-Gaussian. Thus \( \mathbf{E}(X_1 X_1^\top) = \identite_p \). We say that a random vector $X \in \mathbf{R}^p$ is $1$-sub-Gaussian if and only if for all $u \in \mathbf{R}^p$ we have that $\mathbf{E}(\exp(\langle u , X \rangle)) \leq \exp(\|u\|^2/2)$.
\end{assumption}

\begin{assumption}
\label{assmp3}
 The noise components $(\xi_i)_{i=1,\dots,n}$ are \iid centered with unit variance. We assume that the distribution of the noise $P_\xi$ belongs to a set $\mathcal{P}_{\ell,a}$ for some $\ell \geq 2$ where:
\[ \mathcal{P}_{\ell,a} = \{P_\xi \text{ such that } \mathbf{E}(|\xi|^\ell) \leq a^\ell \}. \]
When $\ell = \infty$, we assume instead that $\mathbf{P}(|\xi| \geq t) \leq 2 \exp\left(-\frac{t^2}{2a^2}\right)$. We also suppose that we know $\ell$ or its lower bound and that $a$ is bounded by a sufficiently large constant.
\end{assumption}

\begin{remark}
    Assumption \ref{assmp3}  was used recently in \cite{minsker2024robust} and  is well-suited for handling  heavy-tailed noise. Observe that the class $\mathcal{P}_{2,1}$ is simply the class of centered noise with finite variance. 
    
\end{remark}

In what follows, we define the augmented set of outliers: 
\begin{equation}\label{assump4}
    O':=\left \{i : (Y_i-X_i^T\beta^*)^2 \geq \frac{1}{2 \tau} \right\}\cup O,
\end{equation}
where $\tau$ is the same temperature parameter used in the Welsch loss. The set $O'$ includes the original corrupted samples and the samples with large noise values, that we shall treat as outliers. We denote by $o':=|O'|$.

\subsection{On the loss landscape}\label{subsec:alpha}

\begin{figure}[!t]
    \centering
        \scalebox{1}{%
    \begin{tikzpicture}
        \begin{axis}[
            width=11cm,
            height=7cm,
            xlabel={$\beta_1$},
            ylabel={$f\left((.,\beta_2) \right)$},
            grid=both,
            major grid style={gray!50},
            minor grid style={gray!20},
            tick style={black},
            axis x line=bottom,
            axis y line=left,
            xmin=0, xmax=15,
            ymin=3, ymax=4.2,
            legend style={at={(0.5,-0.15)},anchor=north,legend columns=-1},
        grid=major,
        ]
        
        \addplot[
            thick, orange, smooth
        ] table[
            x=x, y=y,
            col sep=comma
        ] {csv_files/alpha_divergence_data.csv};
        
        \addplot[
            only marks, red, mark=*, mark size=2pt
        ] coordinates {(3.1641641641641645, 3.692167368780044)};
        \node[anchor=south] at (axis cs:3.5,3.62) {Local minimum};

\addplot [
    fill=green!50, fill opacity=0.5,
    cycle list name=exotic, 
    forget plot
] table [
    x=x, y=y,
    col sep=comma
] {csv_files/basin_of_attraction_data.csv} ;

\node[anchor=center, text=black, rounded corners] 
  at (axis cs:11.5,3.1) {Basin of attraction};
        
        \end{axis}
    \end{tikzpicture}}
    \caption{\small Landscape of the empirical Welsch loss function.}
    \label{min_loc_min_glob}
\end{figure}

To gain deeper insights into the properties of the Welsch estimator, we first analyze the structure of its loss landscape. Notably, the objective function to be minimized consists of a sum of non-convex terms, which introduces substantial optimization challenges. Figure \ref{min_loc_min_glob} illustrates the projection of the Welsch loss function onto a single coordinate when $p=2$ and data generated according to model \eqref{model}. Importantly, when large outliers are introduced, the region of local minima gets wider and shifts farther away from the global minimum, making global optimization more challenging. In other words, Figure \ref{min_loc_min_glob} highlights the risk of the optimization process becoming trapped in local minima, potentially leading to suboptimal estimates. To address this, we demonstrate that, with high probability, the objective function is strictly convex within a well-defined basin of attraction.

\begin{theorem} 
\label{conv}
Let the function $f(.)$ be defined as follows:
\begin{equation}
\label{function_f}
    f(\beta) := \frac{1}{\tau n} \sum_{i=1}^n \left(1 - \exp\left(-\frac{\tau}{2}(Y_i - X_i^T \beta)^2\right)\right),
\end{equation}
where $\beta \in \mathbf{R}^{p}$ is a parameter vector, and observations $(Y_1,X_1),\cdots (Y_n,X_n)$ satisfy Assumptions \ref{assmp1}-\ref{assmp3}. Then the function $f$ is strictly convex, with probability at least $1 - 2 \exp(-Dn/C^2)$, on the set $\mathscr{O}_{\tau}$
\begin{equation}
 \mathscr{O}_{\tau}:= \left\{ \beta \in \mathbf{R}^{p}, \frac{1}{n} \sum_{i=1}^n \mathds{1}_{\left\{ \exp\left( \frac{-\tau \left(Y_i - X_i^T \beta \right)^2 }{2} \right) \geq \exp(-\frac{1}{4})\right\} }  \geq D \right\},
\label{set_O}
\end{equation}
where the constant $D$ satisfies the inequality:
\begin{equation}
    \label{cond_D}
    p + 2o'\left(1 + \log\left(\frac{n}{2o'}\right)\right) < \frac{D n}{C^2},
\end{equation}
for a sufficiently large $C$. 

\end{theorem}
\begin{remark}
The condition imposed on $D$ (and equivalently on $C$) is required in order to ensure that the number of unperturbed (or clean) observations exceeds the number of outliers.
\end{remark}
Theorem \ref{conv} indicates that the  Welsch loss function is strictly convex on a well defined set $\mathscr{O}_{\tau}$. Examining the second-order derivative of the function $f$ defined in \eqref{function_f}, weights of the form $w_i(\beta):= \exp\left( \frac{-\tau \left(Y_i - X_i^T \beta \right)^2 }{2} \right)$ emerge, each corresponding to an individual observation. Consequently, the basin of attraction $\mathscr{O}_{\tau}$ is characterized by the vectors for which the weights $(w_i)$ exhibit a degree of uniformity, thereby minimizing the influence of a limited set of points that may be considered outliers. The only issue with the set $\mathscr{O}_{\tau}$ is that it is not necessarily convex and hence minimizing over it does not always lead to a global minimum. The next proposition will address this issue by restricting our attention to a subset that is convex.

\begin{proposition}
    \label{convex_set}
Let $c, \delta > 0$ be sufficiently small constants and $\ell\geq 2$. Consider the convex set : 
\[
\mathscr{J}_c := \{ \beta \in \mathbf{R}^p : \lVert \beta - \beta^* \rVert \leq c \}.
\]  
If $n \geq C_1(o + \log(1/\delta)+p)$ then $\mathscr{J}_c \subset \mathscr{O}_{\tau}$, with probability at least $1-\delta$, for the choice $\tau = C_2 \left(\frac{o + \log(1/\delta)}{n}\right)^{2/\ell}$ where $C_1,C_2$ are large constants. 
\end{proposition}


By combining the results of Theorem \ref{conv} and Proposition \ref{convex_set}, we establish that the Welsch loss function is strictly convex over the convex set $\mathscr{J}_c$. Consequently, any convex optimization procedure used to solve \eqref{estimator} exhibits stable behavior and is guaranteed to converge, provided that the initialization is appropriately chosen, specifically, if the initial vector lies within $\mathscr{J}_c$. In our simulation study, we employ the ``Limited Memory Algorithm for Bound Constrained Optimization" as the optimization method, that we shall explain further in the next section. 

\begin{remark}
    Most of the popular robust estimators belong to the set $\mathscr{J}_c$. For instance, the Least Absolute Deviation (LAD) estimator, defined as
  \begin{equation*}
    \hat{\beta}_{\mbox{LAD}}:= \underset{\beta \in \mathbf{R}^{p}}{\arg\min} \sum_{i=1}^n |Y_i-X_i^T\beta |,
\end{equation*}
was analyzed in \cite{pensia2021robust}. It was established there that, with probability at least $1-\exp(-c n)$, the LAD estimator satisfies
$$
    \lVert \hat{\beta}_{\mbox{LAD}}-\beta^*\rVert\leq c,$$
    for some $c>0$. In practice, it is easy to find vectors in $\mathscr{J}_c$.
\end{remark}

In order to identify an effective estimator that can address the optimization problem due to the Welsch loss, we will employ a two-stage procedure. Initially, we compute the LAD estimator as a convex estimator. In the second part of the procedure, we solve the problem defined in  \eqref{estimator} using an optimization algorithm
that we initialize with the LAD estimator. Our methodology yields a good estimator of $\beta^*$ with optimal non-asymptotic and  asymptotic guarantees, as will be stated later. 

\begin{algorithm}
  \caption{ \small Robust $M$-estimation based on the Welsch loss.}
  \label{alpha_div_algo}
  \KwData{$\left( (X_1,Y_1), \cdots, (X_n,Y_n) \right), \tau, \beta_{0},c$}
  \While{$\mbox{median}((|Y_i - X^\top_i \beta_{k-1}|)_{i=1,\dots,n}) \geq c$}{

    \textbf{Compute a step towards the solution of the Least absolute deviation (LAD) estimator} \\
    $\beta_{k} \gets$ one step towards the solution of the problem 
 $\arg\min_{\beta \in \mathbf{R}^p }\sum_{i=1}^n |Y_i - X_i^\top \beta|$,\\ taking as initialization $\beta_{k-1}$ \\ 
 $k\gets k+1$
  }
   $\beta_{\textrm{LAD}} \gets \beta_{k}$\\
    \textbf{Compute the solution of \eqref{estimator} using  $\beta_{\textrm{LAD}}$ as initialization }\\
    $\hat{\beta} \gets$ The solution of \eqref{estimator} where $\beta_{\textrm{LAD}}$ is considered as the initialization for the optimization procedure. \\
  \Return{$\hat{\beta}$}
\end{algorithm}


From now on $\hat{\beta}$ will denote the output of Algorithm \ref{alpha_div_algo}. In practice, a quasi-Newton method is employed to solve the optimization problems of interest, specifically the ``Limited Memory Algorithm for Bound Constrained Optimisation'' (L-BFGS), which is a limited-memory method adapted to non-linear optimization problems with simple constraints on the variables. Despite its design for constrained optimization problems, this algorithm is also highly effective for unconstrained problems. This is due to the fact that it does not require explicit information about the Hessian matrix, which can be challenging or expensive to compute in the case of high-dimensional problems \citep{Limited_memory_B}. Furthermore, this method incorporates a memory-limited quasi-Newton update step to approximate the Hessian matrix, thereby ensuring that the memory requirement is linear in $n$ \citep{BFGS-B_Z}. The use of this approach for the calculation of the estimator is noteworthy, as it incorporates second-order information via the estimation of the Hessian, thereby enhancing the estimator's robustness to outliers. Additionally, the least absolute deviation estimator (LAD) is employed in the initial phase of the algorithm, as outlined in Algorithm \ref{alpha_div_algo}.

\subsection{Non-asymptotic optimality}\label{subsec:nonasymp}
In this section, we will analyze the two-stage procedure   outlined in Algorithm \ref{alpha_div_algo}.  We recall the reader  that well-known estimators, such as Huber and LAD, fall within the specified basin of attraction defined in $\mathscr{J}_c$ with high probability, thereby ensuring the convergence of the second phase of the algorithm to a global minimum.

\begin{theorem}
\label{deviation_bound_thm}
    Let $\tau$ be chosen such that $n \geq C(o'+p)$. If Assumptions \ref{assmp1},\ref{assmp2},\ref{assmp3} hold, then any solution $\hat{\beta}$ of \eqref{estimator} that belongs to $\mathcal{O}_\tau$,  satisfies the following deviation bound, with probability at least $1- \delta$:
\begin{align*}
        \lVert \hat{\beta}-\beta^*\rVert \leq  C_1 \left( \frac{1}{\sqrt{\tau}}\frac{2o'}{ n} \sqrt{\log\left(\frac{en}{2o'}\right)}+  \sqrt{\frac{p}{n}}+\sqrt{\frac{\log(1/\delta)}{n}} \right),
\end{align*}
where $C_1 >0$ is an absolute constant.
\end{theorem}

The bound derived in Theorem \ref{deviation_bound_thm} highlights an important trade-off when selecting the parameter $\tau$. Specifically, $\tau$ should be chosen to be small enough (and hence $o'$ large enough) to facilitate effective data filtration, which means that the `outliers' set should primarily capture genuinely contaminated observations while minimizing the inclusion of clean data. However, if $\tau$ is too small, the resulting bound may become excessively large, undermining the estimator's performance. In practice, this hyperparameter is typically determined by cross-validation. The next proposition provides the optimal choice of $\tau$ in theory. 
\begin{proposition}
\label{born}
Let $n \geq C_1(o + \log(1/\delta)+p)$. Setting $\tau = C\left(\frac{o + \log(1/\delta)}{n}\right)^{2/\ell}$, and using Assumptions \ref{assmp1},\ref{assmp2},\ref{assmp3}, then, with probability at least $1-\delta$, Algorithm \ref{alpha_div_algo} converges to a unique $\hat{\beta}$ and we get that
$$
 \lVert \hat{\beta}-\beta^*\rVert \leq  C_1 \left( \left(\frac{o}{ n} \right)^{1-1/\ell}\sqrt{\log\left(\frac{en}{2o}\right)}+  \sqrt{\frac{p}{n}}+\sqrt{\frac{\log(1/\delta)}{n}\log\left(\frac{en}{2\log(1/\delta)}\right)} \right).
 $$
\end{proposition}

\begin{remark}
    Algorithm 1 can also handle adversarial contamination in the covariates by incorporating the iterative filtering procedure of \cite{diakonikolas2019recent}. After filtering, the design matrix satisfies weak stability, as in \citet{pensia2021robust}, and all our guarantees, including Proposition \ref{born}, remain valid.
\end{remark}
    The bound in Proposition \ref{born} can be interpreted as follows. The term $\sqrt{\frac{p}{n}} $ corresponds to the parametric rate for estimating a vector of dimension $p$ given $n$ observations. The term $\sqrt{\frac{\log(1/\delta)}{n}} $ corresponds to sub-Gaussian deviations in order to get a bound with probability $1-\delta$. Finally, the term $\left(\frac{o}{ n} \right)^{1-1/\ell}\sqrt{\log\left(\frac{en}{2o}\right)}$ corresponds  to the contamination term. This measures the impact of outliers on the estimation error, capturing how their presence affects the algorithm's robustness.
    This contamination term is minimax optimal as proved in \cite{minsker2024robust}. The extra logarithmic factor, under heavy tails, results from the Assumption that the outliers are generated by an adversary who has access to the entire dataset and to the joint distribution of all variables under model \eqref{model}. Hence Proposition \ref{born} shows that the Welsch estimator is a sub-Gaussian estimator and moreover is minimax optimal. It now remains to show that the latter estimator is superior compared to other robust methods such that Huber's estimator for instance. The next theorem  claims that, given large outliers, the  Welsch estimator discards outlier's contribution.

    \begin{theorem}\label{debiais}
       Let $n \geq C_1(o + \log(1/\delta)+p)$. Setting $\tau = C\frac{\log(1/\delta)}{n}$ and using the same Assumptions in Theorem \ref{deviation_bound_thm}, where we assume further that $\sqrt{\tau}\underset{i \in O}{\min} |\theta_i| \geq C_1(\sqrt{p} + \sqrt{\log(n)} +\sqrt{\log(1/\delta)})$. Then, with probability at least $1-\delta$, we get that
$$
 \lVert \hat{\beta}-\beta^*\rVert \leq  C_2 \left(  \sqrt{\frac{p}{n}}+\sqrt{\frac{\log(1/\delta)}{n}} \right).
 $$
    \end{theorem}
The condition on the outliers magnitude 
in Theorem \ref{debiais} is not optimal and can be relaxed, but it gives a practical illustration of the fact that the Welsch estimator is able to estimate the regression vector $\beta^*$ without the outliers contribution whenever these are too large. To some extent, this procedure adapts to cases where outliers are too large and removes the corresponding bias whenever possible. 


\subsection{Asymptotic efficiency}\label{subsec:asymp}
 After claiming that the  Welsch estimator is non-asymptotically minimax optimal, the objective of this section is to show that the same estimator is also asymptotically efficient. Since we now consider the case where $n$ tends to infinity, it is no longer appropriate to assume the presence of outliers in the observations. Consequently,  we set $\theta=0$ in model \eqref{model}. Theorem  \ref{assymp}, next, establishes the asymptotic efficiency of the estimator given by Algorithm \ref{alpha_div_algo} for a suitably chosen $\tau_n$. 

\begin{theorem}
\label{assymp}
Let the parameter $\tau_n$ in the definition of estimator \eqref{estimator} be chosen as $\tau_n:=\frac{u_n}{n}$, where $u_n$ is a sequence that slowly goes to infinity as $n$ increases. Under the same Assumptions as in Theorem \ref{deviation_bound_thm}  and this choice of $\tau_n$, the estimator given by Algorithm \ref{alpha_div_algo} satisfies the following asymptotic normality result:
\[
\sqrt{n}(\hat{\beta} - \beta^*) \xrightarrow{d} \mathcal{N}(0, \identite_p).
\]
    
\end{theorem}

\section{Numerical experiments}\label{sec:simu}
In this section, we present some empirical results obtained using the methodology outlined in this article. We begin by discussing the results of experiments conducted on simulated data and then present results derived from a real-world dataset.
\subsection{Simulated data}
The experiments were conducted on simulated data generated according to the model described in \eqref{model}. To incorporate heavy-tailed noise into the data, we modeled $\xi$ as a centered Pareto-distributed random variable.

\begin{figure}[!t]
    \centering
    \includegraphics[width=10cm, height = 6cm]{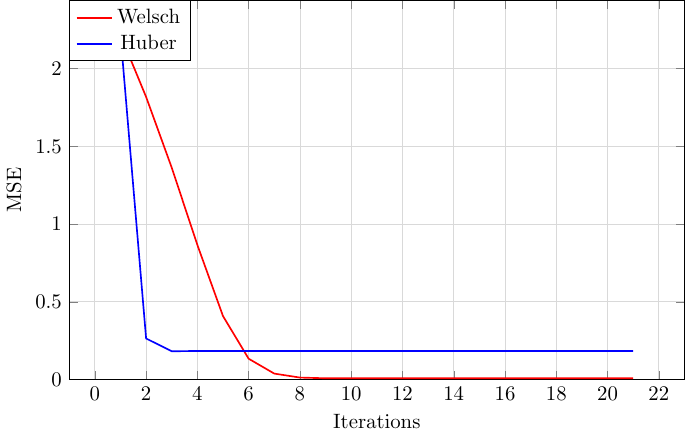}
    \caption{\small Comparison of the speed of convergence between Huber and Welsch estimators.}
    \label{CV_speed} 
\end{figure}

We first compare the convergence speed of our method to that of the well-established Huber-based algorithm, as shown in Figure \ref{CV_speed}. For both approaches, the optimization process is initialized within the basin of attraction  using the LAD estimator, and gradient descent is employed to iteratively minimize the objective function. At each iteration, we assess the deviation between the true parameter $\beta^*$  and its estimator. In this experiment, we use simply gradient descent rather than the L-BFGS algorithm, given that the initialization is already within the basin of attraction. While the Huber estimator exhibits slightly faster convergence, the Welsch estimator achieves a lower 
$\ell_2$  estimation error within a relatively small number of iterations, highlighting its robustness and accuracy.

\begin{figure}[!t]
    \centering
\includegraphics[width=10cm, height = 6cm]{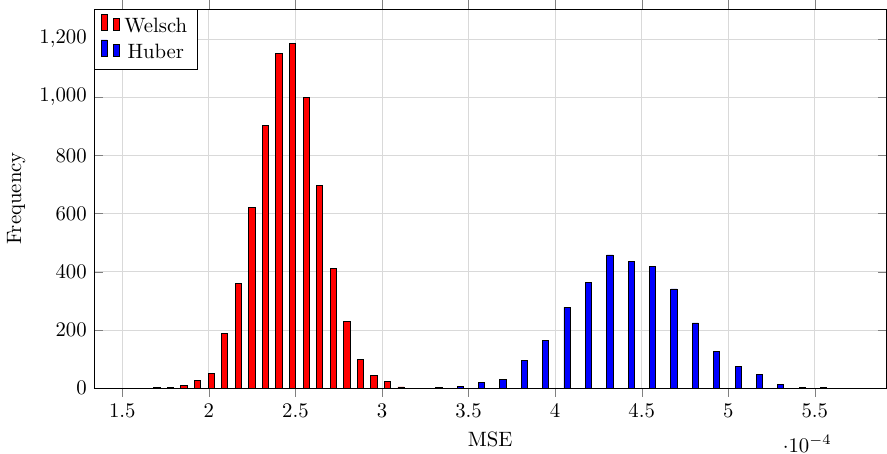}
    \caption{\small Distribution of the MSE of Huber and Welsch estimators under corruption.}
    \label{error}
\end{figure}

In the following experiment, illustrated in Figure \ref{error}, we compare the distribution of the Mean Squared Error (MSE) produced by Welsch and Huber estimators. Specifically, we apply Algorithm \ref{alpha_div_algo} to estimate $\beta^*$ over 10,000 experiments, each using a distinct dataset $\left((X_1, Y_1), \ldots, (X_n, Y_n)\right)$, injecting $10\%$ of outliers adevrsially, generated according to model \eqref{model}. 
In all these experiments, we use the L-BFGS algorithm to solve the optimization problems. We observe that Welsch estimator highly reduces the estimation bias present in the Huber-based estimator, thereby improving the overall estimator performance. 
Additionally, despite our approach consisting of a two-step procedure, the average number of iterations required for convergence is comparable between the two estimators. Finally, To further evaluate robustness of the Weslch estimator, we repeat the experiment with an increased proportion of adversarially injected outliers (20\%). As shown in Figure \ref{fig::error_smooth}, the Welsch estimator continues to exhibit a slight advantage over other redescending M-estimators in terms of error reduction.

\begin{figure}[!t]
    \centering
\includegraphics[width=10cm, height = 6cm]{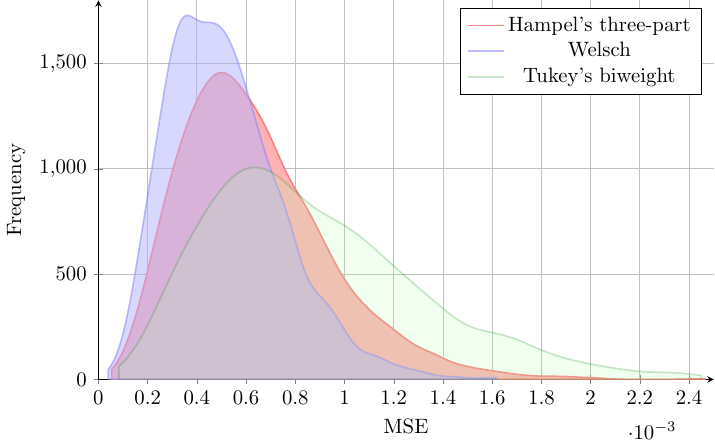}
    \caption{\small Distribution of the MSE of redescending estimators  under heavy corruption.}
    \label{fig::error_smooth}
\end{figure}  

\subsection{Real-world data}
To evaluate the performance of our algorithm on real-world data, we conducted experiments using the Housing dataset \citep{de2011ames}. This dataset contains approximately 79 features describing residential properties and is widely used for predicting house sale prices in Ames, Iowa, USA. As noted in \cite{de2011ames}, five of the 2,930 observations are identified as outliers: three correspond to partial sales that do not accurately reflect market values, while the remaining two represent exceptionally large properties with relatively reasonable prices.

The second real-world dataset used in our evaluation is the Abalone dataset \citep{abalone_1}, which consists of 4,177 samples. It includes eight input variables such as sex, length, diameter, height, total weight, and shell weight, along with a target variable representing the number of rings. The Abalone dataset has been extensively utilized in regression tasks \citep{yang2011feature,gu2014incremental,zou2013generalization} as well as in various machine learning applications \citep{zhou2023deep,demirkaya2024optimal,zhu2023classification}. 

For all studied estimators, hyperparameters were selected in a data-driven manner using cross-validation. To enhance robustness against outliers in the validation set, we used the median rather than the mean in the cross-validation process. Notably, the selected $\gamma$ for Huber regression, obtained through this approach, closely aligns with the rule-of-thumb estimate described in \cite{zhou2024enveloped}. The results presented in Figure \ref{fig::real_world_datasets} are consistent with those obtained in the simulated experiments and discussed in the related work section \ref{alpha_div_section}, hence supporting our theoretical findings. In particular, the Welsch estimator reduces estimation bias and decreases the variance of the residuals compared to state-of-the-art robust estimators. Furthermore, it offers a modest but consistent improvement over other redescenders, as illustrated in Figure~\ref{fig::real_world_datasets_smooth}.

    \begin{figure}[!t]
    \centering
    \begin{subfigure}[b]{0.48\textwidth}
        \centering
       \scalebox{0.7}{%
    \begin{tikzpicture}
\begin{axis}[
    title={},
    xlabel={$y - X^T \hat{\beta}$},
    ylabel={Density},
    enlargelimits=0.15,
    legend style={at={(0.05,0.95)},anchor=north west},
    legend cell align={left},
    grid=both,
    width=12cm,
    height=8cm,
    xmin=-4.5 1e5,
    xmax=2.5 1e5,
    ymax=2 1e-6,
    axis x line=bottom, 
    axis y line=left, 
    legend image post style={rotate=-30} 
]

\addplot[red!50, thick, smooth, fill=red!30, opacity=0.5] table [
    x={Bin Center}, 
    y={Density}, 
    col sep=comma
] {csv_files/density_Huber_armes_housing.csv} \closedcycle;
\addlegendentry{Huber}

\addplot[blue!50, thick, smooth, fill=blue!30, opacity=0.5] table [
    x={Bin Center}, 
    y={Density}, 
    col sep=comma
] {csv_files/density_Alpha-div_armes_housing.csv} \closedcycle;
\addlegendentry{Welsch}

\addplot[green!50!black, thick, smooth, fill=green!30, opacity=0.5] table [
    x={Bin Center}, 
    y={Density}, 
    col sep=comma
] {csv_files/density_QR_armes_housing.csv} \closedcycle;
\addlegendentry{Least absolute deviations}

\draw [red, thick, dashed] (axis cs:0,0) -- (axis cs:0,\pgfkeysvalueof{/pgfplots/ymax});

\end{axis}
\end{tikzpicture}}
    \caption{\small Ames Housing dataset.}
    \label{fig::armes_housing_density}
    \end{subfigure}
    \begin{subfigure}[b]{0.48\textwidth}
     \centering
    \scalebox{0.7}{%
    \begin{tikzpicture}
\begin{axis}[
    title={},
    xlabel={$y - X^T \hat{\beta}$},
    ylabel={Density},
    enlargelimits=0.15,
    grid=both,
    width=12cm,
    height=8cm,
    xmax=5,
    xmin=-15,
    xmax=15,
    ymax=0.3,
    axis x line=bottom, 
    axis y line=left, 
]

\addplot[red!50, thick, smooth, fill=red!30, opacity=2] table [
    x={Bin Center}, 
    y={Density}, 
    col sep=comma
] {csv_files/density_Huber_abalone.csv} \closedcycle;

\addplot[blue!50, thick, smooth, fill=blue!30, opacity=0.5] table [
    x={x}, 
    y={density}, 
    col sep=comma
] {csv_files/alpha_div_abylone_KDE.csv} \closedcycle;

\addplot[green!50!black, thick,smooth, fill=green!30,smooth, opacity=0.2] table [
    x={Bin Center}, 
    y={Density}, 
    col sep=comma
] {csv_files/density_QR_abalone.csv} \closedcycle;

\draw [red, thick, dashed] (axis cs:0,0) -- (axis cs:0,\pgfkeysvalueof{/pgfplots/ymax});

\end{axis}
\end{tikzpicture}}
        \caption{Abalone dataset.}
        \label{fig::Abalone_residuals_density}
    \end{subfigure}
    \caption{\small Distribution of the residuals of the Welsch estimator compared to non-redescenders.} 
    \label{fig::real_world_datasets}
\end{figure}

\begin{figure}[H]
    \centering
   \begin{subfigure}[b]{0.48\textwidth}
        \centering
       \scalebox{0.7}{%
    \begin{tikzpicture}
\begin{axis}[
    title={},
    xlabel={$y - X^T \hat{\beta}$},
    ylabel={Density},
    enlargelimits=0.15,
    legend style={at={(0.0,0.95)},anchor=north west},
    legend cell align={left},
    grid=both,
    width=12cm,
    height=8cm,
    xmin=-1.9 1e5,
    xmax=1.9 1e5,
    ymax=2 1e-6,
    axis x line=bottom, 
    axis y line=left, 
    legend image post style={rotate=-30} 
]

\addplot[red!50, thick, smooth, fill=red!30, opacity=0.5] table [
    x={x}, 
    y={density}, 
    col sep=comma
] {csv_files/armes_housing_hampel.csv} \closedcycle;
\addlegendentry{Hampel’s three-part}

\addplot[blue!50, thick, smooth, fill=blue!30, opacity=0.5] table [
    x={x}, 
    y={density}, 
    col sep=comma
] {csv_files/armes_housing_welsch.csv} \closedcycle;
\addlegendentry{Welsch}

\addplot[green!50!black, thick, smooth, fill=green!30, opacity=0.5] table [
    x={x}, 
    y={density}, 
    col sep=comma
] {csv_files/armes_housing_tukey.csv} \closedcycle;
\addlegendentry{Tukey’s biweight}

\draw [red, thick, dashed] (axis cs:0,0) -- (axis cs:0,\pgfkeysvalueof{/pgfplots/ymax});

\end{axis}
\end{tikzpicture}}
    \caption{\small Ames Housing dataset.}
    \label{fig::armes_housing_density_smooth}
    \end{subfigure}
    \begin{subfigure}[b]{0.48\textwidth}
        \centering
    \scalebox{0.7}{%
    \begin{tikzpicture}
\begin{axis}[
    title={},
    xlabel={$y - X^T \hat{\beta}$},
    ylabel={Density},
    enlargelimits=0.15,
    grid=both,
    width=12cm,
    height=8cm,
    xmax=5,
    xmin=-15,
    xmax=15,
    ymax=0.3,
    axis x line=bottom, 
    axis y line=left, 
]

\addplot[red!50, thick, smooth, fill=red!30, opacity=2] table [
    x={x}, 
    y={density}, 
    col sep=comma
] {csv_files/Hampel_abylone_KDE.csv} \closedcycle;

\addplot[blue!50, thick, smooth, fill=blue!30, opacity=0.5] table [
    x={x}, 
    y={density}, 
    col sep=comma
] {csv_files/alpha_div_abylone_KDE.csv} \closedcycle;

\addplot[green!50!black, thick,smooth, fill=green!30,smooth, opacity=0.2] table [
    x={x}, 
    y={density}, 
    col sep=comma
] {csv_files/tukey_abylone_KDE.csv} \closedcycle;

\draw [red, thick, dashed] (axis cs:0,0) -- (axis cs:0,\pgfkeysvalueof{/pgfplots/ymax});

\end{axis}
\end{tikzpicture}}
        \caption{Abalone dataset.}
        \label{fig::Abalone_residuals_density_smooth}
    \end{subfigure}
    \caption{\small Distribution of the residuals of the Welsch estimator compared to other redescenders.}
    \label{fig::real_world_datasets_smooth}
\end{figure}

\bibliographystyle{chicago}
\bibliography{references.bib}

\newpage

\appendix 
\pagenumbering{arabic}
\renewcommand{\theequation}{\thesection.\arabic{equation}}
\setcounter{equation}{0}

\begin{center}
\Large  Supplementary Material for ``Robust Regression under Adversarial Contamination: Theory and Algorithms for the Welsch Estimator'' 
\end{center}

\medskip
\medskip
\medskip

\section{Technical results}
In this section, we state the concentration inequalities and technical Lemmas that will be used in the proofs that follow. 

\begin{theorem} 
[\cite{Vershynin_2018}]\label{valeur_propre_1}
    Let \( X \in \mathbf{R}^{n \times p} \) be a random matrix whose rows  \( (X_i)_{i \leq n} \) are independent \( K_i \)-sub-Gaussian random variables in \( \mathbf{R}^p \). There exists a sufficiently large constant \( C_K > 0 \), depending only on the parameter \( K = \max_i K_i \), such that for all \( t \geq 1 \), with probability at least \( 1 - e^{-t} \), the following inequality holds:
\[
\left\| \frac{1}{n} XX^T - \frac{1}{n}\mathbf{E}(XX^T) \right\|_{op} \leq \frac{C_K}{n} \left\| \mathbf{E}(XX^T) \right\|_{op} 
\left( \sqrt{\frac{p}{n}} + \sqrt{\frac{t}{n}}+ \frac{p}{n} + \frac{t}{n} \right).
\]
This inequality is sharp when $X$ is isotropic, i.e when $\mathbf{E}[X_iX_i^T]= \identite_p$, for all $0 \leq i \leq n$.
\end{theorem}

The following Lemma is easily derived from the Theorem stated above. 

\begin{lemma} 
[\cite{Vershynin_2018}]\label{valeurs_singulière}
    Let \( X \in \mathbf{R}^{n \times p} \) be a random matrix whose rows  \( (X_i)_{i \leq n} \) are independent \( K_i \)-sub-Gaussian random variables in \( \mathbf{R}^p \). Let \( s_1(X), \dots, s_n(X) \) be the singular values of $ X $ such that \( s_1(X) \geq \dots \geq s_n(X) \). Thus, for all \( \delta > 0 \), there exists a sufficiently large constant \( C_K > 0 \) and a sufficiently small constant $c_K$, depending only on the parameter \( K = \max_i K_i \), such that with probability at least  \( 1 - \delta \), the following inequality holds:
\[
\sqrt{n} - C_K \sqrt{p} - \sqrt{\frac{\log(2/\delta)}{c_K}} \leq s_n(X) \leq s_1(X) \leq \sqrt{n} + C_K \sqrt{p} + \sqrt{\frac{\log(2/\delta)}{c_K}}.
\]
\end{lemma}

\begin{theorem} 
[\cite{Vershynin_2018}]\label{norme}
Let $X$ be an isotropic $K$-sub-Gaussian random vector, then :
\[
\mathbf{E}(\|X\|) \leq 4\sqrt{Kp}.
\]
Moreover, with probability at least \( 1 - \delta \) for any  \( \delta \in (0, 1) \) :
\[
\|X\|\leq 4\sqrt{K p} + 2\sqrt{K\log\left(\frac{1}{\delta}\right)}.
\]
\end{theorem}

\begin{lemma} [\cite{minsker2024robust}]
    \label{control_noise} 
    Let us define the notation \(\xi_{(i)}\) as the \(i\)-th largest component, in absolute value, of the vector \(\xi\). If the noise vector \(\xi\) satisfies Assumption \ref{assmp3} and \(o \leq n/1000\), then we have that

\[
\mathbf{P} \left(  \sum_{i=o}^{n} |\xi|_{(i)}^2 \leq 2n \right) \geq 1 - 2e^{-co},
\]
for an absolute constant \(c > 0\).
\end{lemma}

\begin{lemma}[\cite{minsker2024robust}]
\label{noise_mu}
    For all \(1 \leq i \leq n\), we define \(\mu_i = \sqrt{\frac{C}{n}} \left(\frac{n}{i}\right)^{1/\ell}\) for \(\ell \geq 2\) and \(C \geq 80\). Then, for all  \(k \geq 1\), the following inequality holds :

\[
\mathbf{P}\left( \max_{i \geq k} \frac{|\xi|_{(i)}}{\sqrt{n} \mu_i} \geq \frac{1}{20} \right) \leq 2e^{-k}.
\]
\end{lemma}

\begin{lemma}\label{lem:o'}
   Let $\ell \geq 2$ and $\tau = C\left( \frac{o + \log(1/\delta)}{n}\right)^{2/\ell} $ for some $C>0$ large enough. Then, with probability at least $1-\delta$, $o' \leq o+\log(1/\delta)$.
\end{lemma}
\begin{proof}
Notice that: 
\begin{align*}
    o'&=  |O'| =\left| \left \{i : (Y_i-X_i^T\beta^*)^2 \geq \frac{1}{2 \tau} \right\}\cup O \right|\\
    &= |O|+ \left| \left \{i : (Y_i-X_i^T\beta^*)^2 \geq \frac{1}{2 \tau} \right\} \setminus O \right|\\
    &\leq o+ \sum_{i=0}^n \mathds{1}_{ \{ \xi_i^2\geq \frac{1}{2 \tau}  \}}. 
\end{align*}
Moreover:
\begin{align*}
    \mathbf{P} \left( \sum_{i=0}^n \mathds{1}_{ \{ \xi_i^2\geq \frac{1}{2 \tau}  \}} \leq \log(1/\delta)\right)&=1-  \mathbf{P} \left( \sum_{i=0}^n \mathds{1}_{ \{ |\xi|_{(i)}^2\geq \frac{1}{2 \tau}  \}} \geq \log(1/\delta)\right)\\
    & \geq 1- \mathbf{P}\left(|\xi|_{\log\left(1/\delta\right)} \geq \sqrt{\frac{1}{2\tau} }\right)\\
    &\geq 1-\delta.
\end{align*}
The final bound is obtained using Lemma \ref{noise_mu}.
\end{proof}

\section{Proof of Theorem \ref{conv}}

 To establish the strong convexity of the objective function $f$, we analyze its second order derivative. First, we notice the following equivalence: 
     \[
     \mathscr{O}_{\tau}:=\left\{ \beta \in \mathbf{R}^{p}, \frac{1}{n} \sum_{i=1}^n \mathds{1}_{\left\{ w_i (\beta) \geq \exp(-\frac{1}{4})\right\} }  \geq D \right\} \Leftrightarrow  \left\{ \beta \in \mathbf{R}^{p}, \frac{1}{n} \sum_{i=1}^n \mathds{1}_{\left\{ \tau (Y_i-X_i\beta)^2 \leq \frac{1}{2}\right\} }  \geq D 
     \right\},     \] 
where the constant $D$ satisfies the inequality:
\begin{equation*}
    o'\left(1 + \log\left(\frac{n}{o'}\right)\right) < \frac{D n}{C^2},
\end{equation*}
for a sufficiently large constant $C$.  In fact, this basin of attraction ensures that a significant number of observations remain ``uncorrupted''. Specifically, any vector  $\beta$ in  $\mathscr{O}_{\tau}$, can be seen as candidate estimator for $\beta^*$. This Assumption is generally not hard to verify in practice. Writing down the second order derivative of  $f$ in \eqref{function_f}, for $\beta \in \mathscr{O}_{\tau}$, we get the following: 

    \begin{equation}
        \label{nabla}
        \nabla^2 f(\beta) \geq    \left( \frac{\exp(-1/4)}{2n} \min_{S: |S|\geq D n}\lambda_{\min}(X_SX_S^T) - \frac{2 \exp(-3/2)}{n} \max_{U: |U|\leq 2o'}\lambda_{\max} (X_U X_U^T)\right) \identite_p.
    \end{equation}
In fact: 
\begin{align}
    \nabla^2 f(\beta)&:= \frac{1}{n} \sum_{i=1}^n w_i(\beta) X_iX_i^T - \frac{\tau}{n} \sum_{i=1}^n X_i X_i^T(Y_i-X_i^T\beta)^2 w_i(\beta)\\
    &= \frac{1}{n} \sum_{i=1}^n w_i(\beta) X_iX_i^T \mathds{1}_{\left\{ \tau (Y_i-X_i\beta)^2 \leq \frac{1}{2}\right\} } \\
    &+ \frac{1}{n}  \sum_{i=1}^n  w_i(\beta) X_i X_i^T \left( 1- \tau (Y_i-X_i^T\beta)^2  \right) \mathds{1}_{\left\{ \tau (Y_i-X_i\beta)^2 \geq \frac{1}{2}\right\} }\\
    &- \frac{1}{n} \sum_{i=1}^n X_i X_i^T \underbrace{\tau (Y_i-X_i^T\beta)^2}_{\leq \frac{1}{2}} 
    w_i(\beta) \mathds{1}_{\left\{ \tau (Y_i-X_i\beta)^2 \leq \frac{1}{2}\right\} }\\
    & \geq \frac{\exp(-1/4)}{2n} X_{I}X_{I}^T- \frac{2 \exp(-3/2)}{n}  X_{I^{c}}X_{I^{c}}^T,
    \; \;\text{where $I:= \left\{i:  \tau (Y_i-X_i\beta)^2 \leq \frac{1}{2}\right\}$.} \label{eq:proof:3}\\
    & \geq  \left( \frac{\exp(-1/4)}{2n} \lambda_{\min}(X_IX_I^T) - \frac{2 \exp(-3/2)}{n}\lambda_{\max} (X_{I^c} X_{I^c}^T)\right) \identite_p\\
    &\geq  \left( \frac{\exp(-1/4)}{2n} \min_{S: |S|\geq D n}\lambda_{\min}(X_SX_S^T) - \frac{2 \exp(-3/2)}{n} \max_{U: |U|\leq 2o'}\lambda_{\max} (X_U X_U^T)\right) \identite_p \label{eq:proof:4},
\end{align}
where we have used the following inequality in \eqref{eq:proof:3}:
\[
 \forall x \in \mathbf{R}^+,
(1-x)e^{-\frac{x}{2}}\geq -2 \exp(-3/2),
\]
while inequality \eqref{eq:proof:4} is true since: 
\[
\beta \in \mathscr{O}_{\tau} \Rightarrow |I|=\sum_{i=1}^n \mathds{1}_{\left\{ \tau (Y_i-X_i\beta)^2 \leq \frac{1}{2}\right\} }  \geq D n .
\]
It remains to  control the two empirical processes : $\min_{S: |S|\geq D n } \lambda_{\min}(X_S X_S^T)$ and $\max_{U: |U|\leq  2o'} \lambda_{\max}(X_U X_U^T)$.
We will use Theorem \ref{valeur_propre_1} to do so.
Since $\mathbf{E}[XX^T] = \identite_n$, we have that $\frac{1}{n} XX^T - \mathbf{E}[XX^T] = \frac{1}{n} XX^T - \identite_n$ is symmetric. Using  Theorem \ref{valeur_propre_1}, we know that there exists a sufficiently large constant $C_K > 0$ such that with probability at least $1 - e^{-t}$ :
\[
 \max_{i\leq n} \left| \lambda_i \left(\frac{1}{n} XX^T - I \right)\right|  \leq C_K 
 \max\left( \left(\sqrt{\frac{p}{n}} + \sqrt{\frac{t}{n}}\right), \left( \frac{p}{n} + \frac{t}{n}\right) \right).
\]
Thus with probability $1-e^{-t}$ we have as well: 
\begin{align*}
    1-C_K 
 \max\left( \left(\sqrt{\frac{p}{n}} + \sqrt{\frac{t}{n}}\right), \left( \frac{p}{n} + \frac{t}{n}\right) \right) \leq \frac{1}{n}\lambda_{\min}(XX^T) \leq \frac{1}{n}\lambda_{\max}(XX^T) \leq    1&+ \\
 C_K  \max\left( \left(\sqrt{\frac{p}{n}} + \sqrt{\frac{t}{n}}\right), \left( \frac{p}{n} + \frac{t}{n}\right) \right).
\end{align*}
Since we are treating the dense case, we have the following inequality $1 \gg \sqrt{\frac{p}{n}} \geq \frac{p}{n}$. Moreover, as long as $t < n$, we have: $1 \gg \sqrt{\frac{t}{n}} \geq \frac{t}{n}$, as well.  Consequently, this argument implies that with a probability of at least $1-2e^{-t}$, we have:
\begin{align*}
    n-C_1 
  \left(\sqrt{pn} + \sqrt{nt}\right) \leq \lambda_{\min}(XX^T) \leq \lambda_{\max}(XX^T) \leq    n+
 C_1   \left(\sqrt{np} + \sqrt{nt}\right).
\end{align*}
To extend the previous bound to control the maximum and the minimum we will use a union bound. Without loss of generality, we suppose that $Dn$ is an integer. Otherwise, it is possible to replace $Dn$ by its ceiling in the proof.
\begin{align*}
    &\mathbf{P}\left(\min_{|S|\geq Dn } \frac{ \lambda_{\min}(X_SX_S^T)}{\left( |S| - C(\sqrt{|S|p}+\sqrt{|S|t})\right)}  \leq 1\right)\\
    = &\mathbf{P}\left( \exists S, |S|\geq Dn, \frac{\lambda_{\min}(X_SX_S^T) }{\left( |S| - C(\sqrt{|S|p}+\sqrt{|S|t} \right)}  \leq 1\right)\\
    = &\mathbf{P}\left (\bigcup_{i= Dn }^n  |S|=i,\frac{\lambda_{\min}(X_SX_S^T)}{\left(|S| - C\left( \sqrt{|S|p}+\sqrt{|S|t}\right)\right)}   \leq 1 \right)\\
    \leq &\sum_{i=Dn}^n \binom{n}{i} \mathbf{P}\left( \lambda_{\min}(X_SX_S^T)\leq i-C \left(\sqrt{pi}-\sqrt{ti} \right)
    \right).
    \end{align*}
Since for every subset \(S\) with cardinality at least \(Dn\), the matrix \(X_S X_S^T\) is positive semi-definite, we know that \(\lambda_{\min}(X_S X_S^T) \geq 0\). Therefore, for the probability above to be well-defined, it is required that for every \(i\) satisfying \(Dn \leq i \leq n\), the following inequality holds: 
\[
i - C \left(\sqrt{pi} - \sqrt{ti}\right) \geq 0.
\]
It suffices to verify this condition for \(i = Dn\), which leads to a first constraint on \(D\) and \(t\): 
\[
t \leq \frac{Dn}{C_1^2}.
\]
On the other hand, to control the growth of the binomial term resulting from the union bound, we impose a second condition on \(D\):
\[
2o' \left( 1 + \log \left( \frac{n}{2o'} \right) \right) \leq t \leq \frac{Dn}{C_1^2},
\]
which is equivalent to the condition \eqref{cond_D}.
Finally, using $\sum_{i=k}^n \binom{n}{i} \leq \left( \frac{e n}{n-k} \right)^{n-k}$ we get, moreover, that:
    \begin{align*}
    \mathbf{P}\left( \min_{|S|\geq Dn} \lambda_{\min}(X_SX_S^T) \leq Dn - C\sqrt{np}-\sqrt{t+\frac{2o'}{c}\log\left(\frac{e n}{2o'}\right)}\right)\leq& \exp(-c_1t).
\end{align*}
Following similar steps, we demonstrate that:
\begin{align*}
    \mathbf{P}\left( \max_{|U|\leq 2o'} \lambda_{\max}(X_UX_U^T) \geq 2o' + C\sqrt{2o'p}+\sqrt{t+\frac{2o'}{c}\log\left(\frac{en}{2o'}\right)}\right)\leq& \exp(-c_1t).
\end{align*}
As a consequence, we have established that with high probability at least $1-\exp(-Dn/C^2)$ that:
$
\min_{|S| \geq Dn} \lambda_{\min}(X_S X_S^T)
$ 
is larger than \(\frac{3Dn}{4} - C_1 \sqrt{\frac{3n}{4} p}\), and
$
\max_{|U| \leq 2o'} \lambda_{\max}(X_U X_U^T)
$
is smaller than \(2o' + \sqrt{nD/C^2}+ C_1 \sqrt{\frac{Dn}{4} p}\), for a sufficiently large constant \(C_1\).
Referring back to \eqref{nabla}, when $n$ is sufficiently large, the term \(C_1\sqrt{\frac{Dnp}{4} }\) is small compared to $Dn$. Moreover, the following inequality holds:
$
\frac{\exp(-1/4) D}{2} - \frac{4 o' \exp(-3/2)}{n} > 0,
$
as long as $D > \frac{8o'}{n} \exp(-10/8)$. This condition is satisfied due to the condition \eqref{cond_D}. Thus, we have shown that the function $f$ is strictly convex in the basin of attraction.

\section{Proof of proposition \ref{convex_set}}
Let $\hat{\beta}_c \in \mathscr{J}_c$, where $n-o'-\frac{4n}{5}c \geq Dn$. Notice first: 
\begin{align*}
    \mathds{1}_{\{(Y_i-X_i\hat{\beta}_{c})^2\leq \frac{1}{2 \tau}\}}&=
    \mathds{1}_{\{(\xi_i+X_i(\beta^*-\hat{\beta}_{c}))^2\leq \frac{1}{2 \tau}\}}\\
    &\geq \mathds{1}_{\{ \xi_i^2 + \left( X_i(\beta^* -\hat{\beta}_{c}) \right)^2 \leq \frac{1}{4
\tau}\}}.
\end{align*}
Let's fix $\tau$ and consider, for simplicity, two positive reals $a$ and $b$. Then we have:
\begin{align*}
    \mathds{1}_{ \{ a\leq \frac{1}{6 \tau}  \}} &\leq \mathds{1}_{ \{ a\leq \frac{1}{6 \tau}  \}} \mathds{1}_{ \{ b\leq \frac{5}{4}  \}} + \mathds{1}_{ \{ b\geq \frac{5}{4}  \}}. \  (*)
\end{align*}
Notice that : $\mathds{1}_{ \{ a\leq \frac{1}{6 \tau}  \}} \mathds{1}_{ \{ b\leq \frac{5}{4}  \}}= \mathds{1}_{\left\{\{ a\leq \frac{1}{6 \tau}  \} \cap   \{ b\leq \frac{5}{4}  \} \right\}}$. Then using a collectively exhaustive set of events we get:
\begin{align*}
    \left\{\{ a\leq \frac{1}{6 \tau}  \} \cap   \{ b\leq \frac{5}{4}  \} \right\}&= \left\{\underbrace{\{ a\leq \frac{1}{6 \tau}  \} \cap   \{ b\leq \frac{5}{4}  \} \cap \{ b \leq \frac{a}{2} \} }_{\subset\{a+b \leq \frac{1}{4 \tau} \}}          \right\} \bigcup\\
    &\left\{\underbrace{\{ a\leq \frac{1}{6 \tau}  \} \cap   \{ b\leq \frac{5}{4}  \} \cap \{ b \geq \frac{a}{2} \} }_{\subset\{a+b \leq \frac{15}{4 } \}}          \right\}.
\end{align*}
Thus $     \left\{\{ a\leq \frac{1}{6 \tau}  \} \cap   \{ b\leq \frac{1}{4}  \} \right\} \subset \{a+b \leq \frac{1}{4 \tau} \} \bigcup \{a+b \leq \frac{15}{4 } \} $. Now considering $\tau$ sufficiently small so that: $\frac{1}{4 \tau} \geq \frac{15}{4}$, consequently $\{a+b \leq \frac{15}{4 } \} \subset \{a+b \leq \frac{1}{4 \tau} \} $. Finally, $    \left\{\{ a\leq \frac{1}{6 \tau}  \} \cap   \{ b\leq \frac{5}{4}  \} \right\} \subset \{a+b \leq \frac{1}{4 \tau} \} \Rightarrow \mathds{1}_{ \{ a\leq \frac{1}{6 \tau}  \}} \mathds{1}_{ \{ b\leq \frac{5}{4}  \}} \leq \mathds{1}_{\{a+b \leq \frac{1}{4 \tau} \}}. $ Thus: 
\begin{align}
\label{inter_res_lemma}
\mathds{1}_{ \{ a\leq \frac{1}{6 \tau}  \}}&\leq \mathds{1}_{ \{ a+b\leq \frac{1}{4 \tau}  \}}+ \mathds{1}_{ \{ b\geq \frac{5}{4}  \}}.
\end{align}
Using \eqref{inter_res_lemma} with $a=\xi_i^2$ and $b=\left(X_i\left( \beta^* - \hat{\beta}_{c}\right) \right)^2$, we get: 
\begin{align*}
    & \mathds{1}_{\{(Y_i-X_i\hat{\beta}_{c})^2\leq \frac{1}{2 \tau}\}}\geq 
    \mathds{1}_{ \{ \xi_i^2\leq \frac{1}{6 \tau}  \}}- \mathds{1}_{ \left\{ \left(X_i\left( \beta^* - \hat{\beta}_{c}\right) \right)^2 \geq \frac{5}{4}  \right\} }\\
    &\Rightarrow  |I(\hat{\beta}_{c})| \geq \sum_{i=0}^n \mathds{1}_{ \{ \xi_i^2\leq \frac{1}{6 \tau}  \}} - \sum_{i=0}^n \mathds{1}_{ \left\{ \left(X_i\left( \beta^* - \hat{\beta}_{c}\right) \right)^2 \geq \frac{5}{4}  \right\} }.
\end{align*}
If $\tau = C_2 \left(\frac{o + \log(1/\delta)}{n}\right)^{2/\ell}$, then $\sum_{i=0}^n \mathds{1}_{ \{ \xi_i^2\leq \frac{1}{6 \tau}  \}}=\sum_{i=0}^n\mathds{1}_{ \{ \xi_{(i)}^2\leq \frac{1}{6 \tau}  \}}= \sum_{i=0}^{o'}\mathds{1}_{ \{ \xi_{(i)}^2\leq \frac{1}{6 \tau}  \}} +\sum_{i=o'}^{n}\mathds{1}_{ \{ \xi_{(i)}^2\leq \frac{1}{6 \tau}  \}} \geq \sum_{i=o'}^{n}\mathds{1}_{ \{ \xi_{(i)}^2\leq \frac{1}{6 \tau}  \}}  \geq n-o - \log(1/\delta)$ with probability $1-\delta$, thanks to Lemma \ref{lem:o'}. To control the second term uniformly over $\beta$, we use the following argument: considering two reel numbers having the same sign $a$ and $b$, then we have: $\mathds{1}_{ \{ a \geq b\}} \leq \frac{a}{b}$. Thus $\sum_{i=0}^n \mathds{1}_{ \left\{ \left(X_i\left( \beta^* - \hat{\beta}_{c}\right) \right)^2 \geq \frac{5}{4}  \right\} } \leq \frac{4}{5} \left \rVert X(\beta^*- \hat{\beta}_{c})  \right \rVert_2^2$. Thus with high probability: 
\begin{align*}
    |I(\hat{\beta}_{c})| &\geq n-o'-\frac{4}{5}\lambda_{\max}(XX^T) \lVert \beta^*-\hat{\beta}_{\text{plug-in}} \lVert_2^2 \\
    & \geq n-o'- \frac{4}{5} \left( n+C (\sqrt{np}+\sqrt{tn}) \right) \lVert \beta^*-\hat{\beta}_{c} \lVert_2^2
\end{align*}
Since \(\left( n + C (\sqrt{np} + \sqrt{tn}) \right) \asymp n \) and \(\lVert \hat{\beta}_{c} - \beta^* \rVert_2^2 \leq c \), where $n-o'-\frac{4n}{5}c \geq Dn$  by Assumption. Thus \(  \sum_{i=1}^n \mathds{1}_{\left\{ \exp\left( \frac{-\tau \left(Y_i - X_i^T \hat{\beta}_c \right)^2 }{2} \right) \geq \exp(-\frac{1}{4})\right\} }  \geq Dn \). Consequently \( \hat{\beta}_c \in \mathscr{O}_{\tau} \). Finally $\mathscr{J}_c \subset \mathscr{O}_{\tau}$ for $\tau = C_2 \left(\frac{o + \log(1/\delta)}{n}\right)^{2/\ell}$ with probability at least $1-\delta$.

\section{Proof of Theorem \ref{deviation_bound_thm}}
Going back to the model we are considering \eqref{model}, if we assume that the noise has a heavy-tailed distribution, then it is possible for the noise to take large values. To deal with these large values, we treat them as outliers. Theorem \ref{deviation_bound_thm} states that if the noise satisfies Assumption \ref{assmp3}, then with high probability the number of large values taken by the noise is small. This implies that the large noise values can be treated as outliers without compromising the sparsity of $\theta$, hence the use of the set $O'$. 
Considering Lemma \ref{control_noise}, we treat the $o'$ largest noise components \((\xi_i)_n\), in absolute value, as outliers. We also know that \(o' \geq o\). We denote the new \((n-o')\)-sparse noise  vector by $\xi$. Yet we will consider the following event: 
\[
\Omega:= \left\{ \text{Function $f$ is convex on the set $\mathscr{O}_{\tau}$ }\right\}
\]
Theorem \ref{conv} states that this event happens with probability at least $1-\exp(-cn)$. 
We start by analyzing the first-order condition:
\begin{equation*}
    -\frac{1}{n} \sum_{i=1}^n X_i (Y_i-X_i^T \hat{\beta})  \exp\left(-\frac{\tau}{2}(Y_i-X_i^T \hat{\beta})^{2}\right)=0.\  (*)
\end{equation*}
For the sake of simplicity, we adopt the following notation: $w_i(\hat{\beta}):= \exp\left(-\frac{\tau}{2}(Y_i - X_i^T \hat{\beta})^2\right)$. Returning to equation $(*)$, we get that:
\begin{align*}
\frac{1}{n} \sum_{i=1}^n w_i(\hat{\beta}) X_i X_i^T (\hat{\beta}-\beta^*) &= \frac{1}{n} \sum_{i=1}^n w_i(\hat{\beta}) X_i (\theta_i+\xi_i)\\
\implies \frac{1}{n} (\hat{\beta}-\beta^*)^T \sum_{i=1}^n w_i(\hat{\beta}) X_i X_i^T (\hat{\beta}-\beta^*) &= \frac{1}{n} \sum_{i=1}^n w_i (\hat{\beta}) X_i^T (\hat{\beta}-\beta^*) (\theta_i+\xi_i) \\
\implies \frac{1}{n} (\hat{\beta}-\beta^*)^T \sum_{i=1}^n w_i(\hat{\beta}) X_i X_i^T (\hat{\beta}-\beta^*) &= \frac{1}{n} \sum_{i \in \text{clean}} w_i(\hat{\beta}) X_i^T (\hat{\beta}-\beta^*) \xi_i \\
&\quad + \frac{1}{n} \sum_{i \in \text{outliers}} w_i(\hat{\beta}) X_i^T (\hat{\beta}-\beta^*) \theta_i\\
\implies \frac{1}{n} (\hat{\beta}-\beta^*)^T \sum_{i \in \text{clean}} w_i(\hat{\beta}) X_i X_i^T (\hat{\beta}-\beta^*) &= \frac{1}{n} \sum_{i \in \text{clean}} w_i(\hat{\beta}) X_i^T (\hat{\beta}-\beta^*) \xi_i \\
&\quad + \frac{1}{n} \sum_{i \in \text{outliers}} w_i(\hat{\beta}) X_i^T (\hat{\beta}-\beta^*) \theta_i \\
&\quad - \frac{1}{n} (\hat{\beta}-\beta^*)^T \sum_{i \in \text{outliers}} w_i(\hat{\beta}) X_i X_i^T (\hat{\beta}-\beta^*).
\end{align*}
As mentioned in the Assumptions in Section \ref{sec:theory}, we define the set of outliers as the elements $O'$ such that
$$
    O':=\left \{i : (Y_i-X_i^T\beta^*)^2 \geq \frac{1}{2 \tau} \right\}\cup O.
$$
The proof proceeds in three stages mainly. First, we control the left-hand term; second, we address the corrupted terms; and lastly, we handle the uncontaminated terms.

\paragraph{Left-hand term}
Let $\hat{\beta} \in \mathscr{O}_{\tau}$: 
\begin{align*}
    \frac{1}{n} (\hat{\beta}-\beta^*)^T \sum_{i \in \text{clean}} w_i(\hat{\beta}) X_i X_i^T (\hat{\beta}-\beta^* ) 
    & \geq   \frac{1}{n} (\hat{\beta}-\beta^*)^T \sum_{i \in \text{clean}} w_i(\hat{\beta}) X_i X_i^T (\hat{\beta}-\beta^* ) \mathds{1}_{\left\{ w_i (\beta) \geq \exp(-\frac{1}{4})\right\} } 
    \\
    &\geq \frac{\exp(-1/4)}{n} (\hat{\beta}-\beta^*) X_IX_I^T (\hat{\beta}-\beta^*), \\
    \end{align*}
    where $I:= \left\{i:  \tau (Y_i-X_i\beta)^2 \leq \frac{1}{2}\right\} $. Thus:
    \begin{align*}
       \frac{1}{n} (\hat{\beta}-\beta^*)^T \sum_{i \in \text{clean}} w_i(\hat{\beta}) X_i X_i^T (\hat{\beta}-\beta^* )  &\geq    \lambda_{\min} (X_IX_I^T) \frac{\exp(-1/4)}{n} \lVert \hat{\beta}-\beta^* \lVert^2_2\\
    &\quad\geq  \min_{|S|\geq Dn}  \lambda_{\min} (X_SX_S^T) \frac{\exp(-1/4)}{n} \lVert \hat{\beta}-\beta^* \lVert^2_2.
   \end{align*}
We have shown the first part of the proof stated as follows: 
\begin{equation}
    \label{left_hand_sid_term}
    \frac{1}{n} (\hat{\beta}-\beta^*)^T \sum_{i \in \text{clean}} w_i(\hat{\beta}) X_i X_i^T (\hat{\beta}-\beta^* ) \geq \min_{|S|\geq \frac{3n}{4}}  \lambda_{\min} (X_SX_S^T) \frac{\exp(-1/4)}{n} \lVert \hat{\beta}-\beta^* \lVert^2_2.
\end{equation}

\paragraph{Contaminated observations}

Notice first: 
\begin{align*}
    \forall i \in \text{outliers},\ w_i(\hat{\beta})= \exp(-\frac{\tau}{2} (X_i^T(\beta^*-\hat{\beta})+\theta_i)^2).
\end{align*}
Indeed, since we treat the $o'$ largest realizations of the noise, in absolute value, as contaminated observations, this implies that the outlier vector is now at most $o'$-sparse. Consequently, for all $i$ in the set of outliers, we have: $Y_i = X_i^T \beta^* + \theta_i$. Thus, we can bound the error due to the contaminated observations as follows: 
\begin{equation}
\label{term_outliers}
\begin{aligned}
     &\frac{1}{n} \sum_{i \in \text{outliers}} w_i(\hat{\beta}) X_i^T (\hat{\beta}-\beta^*) \theta_i
 - \frac{1}{n} (\hat{\beta}-\beta^*)^T \sum_{i \in \text{outliers}} w_i(\hat{\beta}) X_i X_i^T (\hat{\beta}-\beta^*) = \\
 &\frac{1}{n} \sum_{i \in \text{outliers}}  X_i^T (\hat{\beta}-\beta^*)w_i(\hat{\beta}) (\theta_i+X_i^T(\beta^*-\hat{\beta}))\leq \\
 &\frac{1}{n} \left \lVert X_O (\hat{\beta}-\beta^*)\right \lVert  
 \left \lVert w(\hat{\beta}) \otimes \left [ X_O  (\beta^*-\hat{\beta}) +\theta\right]\right \lVert .
\end{aligned}
\end{equation}
On the one hand:
\begin{equation}
\label{term_1_outliers}
    \begin{aligned}
         \left \lVert w(\hat{\beta}) \otimes \left [ X_O  (\beta^*-\hat{\beta}) +\theta\right]\right \lVert^2 &=
         \sum_{i\in \text{outliers}} w_i(\hat{\beta})^2 (\theta_i + X_i^T(\beta^*-\hat{\beta}))^2\\
         &=\sum_{i\in \text{outliers}} \exp(-\tau (X_i^T(\beta^*-\hat{\beta})+\theta_i)^2) (\theta_i + X_i^T(\beta^*-\hat{\beta}))^2\\
         &\leq \frac{2 o'}{\tau e}.
    \end{aligned}
\end{equation}
To control the contaminated observations weighted by the exponential terms in the previous sum, we employed the following argument:
\begin{align*}
    \forall x \in \mathbb{R}^{+} ,\ x\exp(-x) \leq \frac{1}{e}.
\end{align*}
On the other hand, we have: 
\begin{equation}
    \label{term_2_outliers}
    \begin{aligned}
 \left \lVert X_O (\hat{\beta}-\beta^*)\right \lVert &\leq \left \lVert X_O \right \lVert_{op} \left \lVert(\hat{\beta}-\beta^*)\right \lVert \\
    &\leq \max_{|O|\leq 2o'} \left \lVert X_O \right \lVert_{op}
    \left \lVert(\hat{\beta}-\beta^*)\right \lVert.
    \end{aligned}
\end{equation}
Now combining (\ref{term_outliers},\ref{term_1_outliers},\ref{term_2_outliers}), we finally get :
\begin{equation}
    \label{outlier_term_final_bound}
    \begin{aligned}
    &\frac{1}{n} \sum_{i \in \text{outliers}} w_i(\hat{\beta}) X_i^T (\hat{\beta}-\beta^*) \theta_i
 - \frac{1}{n} (\hat{\beta}-\beta^*)^T \sum_{i \in \text{outliers}} w_i(\hat{\beta}) X_i X_i^T (\hat{\beta}-\beta^*) \leq \\
 & \max_{|O|\leq 2o'}  \sqrt{\frac{\lambda_{\max}(X_OX_O^T)}{n}}
 \sqrt{\frac{2o'}{\tau e n }} \left \lVert\hat{\beta}-\beta^*\right \lVert  .
    \end{aligned}
\end{equation}

\paragraph{Clean observations}
For this part, we use again the following notation:
\ $\tilde{X} := w(\hat{\beta}) \otimes \mathbf{X}$, where $\otimes$ is the Hadamard (element-wise) product. Additionally, we define $\pi$ as the orthogonal projector onto the column space of $X$.
\begin{equation}
\label{noise term}
    \begin{aligned}
    \sum_{i\in \text{clean}} w_i(\hat{\beta}) X_i^T (\hat{\beta}-\beta^*) \xi_i &=\left \langle \tilde{X}_{C} (\hat{\beta}-\beta^*) , \xi\right\rangle\\
     &\leq \lVert \tilde{X}_{C} (\hat{\beta}-\beta^*)\lVert \lVert \pi \xi \lVert .
\end{aligned}
\end{equation}
In fact: 
\begin{align*}
   | \langle \tilde{X}_{C} (\hat{\beta}-\beta^*) , \xi \rangle| &= 
    |  \langle \tilde{X}_{C} (\hat{\beta}-\beta^*) , \pi\xi \rangle +
     \overbrace{\langle \tilde{X}_{C} (\hat{\beta}-\beta^*) , \xi (\mathbf{I}-\pi) \rangle}^{0}| \leq \lVert \pi \xi\rVert \lVert \tilde{X}_{C} (\hat{\beta}-\beta^*)\rVert.
\end{align*}
Let's first address the term $\lVert \pi \xi \rVert$. Recall that we have the closed-form expression for $\pi$ given by:
\begin{equation}
\label{noise_term_1}
\begin{aligned}
    \lVert \pi \xi \lVert^2&= \lVert X \left( X^TX\right)^\dag X^T \xi \lVert^2= \xi^TX(X^TX)^\dag X^T \xi\\
    &\leq \lambda_{\max}((XX^T)^\dag) \lVert \xi \rVert^2 
    \left\lVert \frac{X^T\xi}{\lVert \xi \rVert} \right\rVert^2\\
    &\stackrel{(P)}{\leq} \lambda_{\min}^{-1} (X^TX) \lVert \xi \rVert^2 
    \lVert u \rVert^2 \text{($u\in \mathbf{R}^{p}$, s.t $u\in SG(I)$)}\\
    &\stackrel{(1-\delta)}{\leq} \frac{2n \lVert u \rVert^2}{\lambda_{\min} (X^TX)}\\
    &\stackrel{(1-2\delta)}{\leq} \frac{2n (4\sqrt{p}+2\log(1/\delta))^2}{\lambda_{\min} (X^TX)}.
\end{aligned}
\end{equation}
The second inequality holds with probability 1. Indeed: 
$X^\top\xi=\sum_{i=1}^n X_i\xi_i$. Since the random vectors $(X_i)_{0 \leq i \leq n}$ are \iid and 1-sub-Gaussian, and $\xi$ is independent of $X$, it follows that for each $i$, $\xi_i X_i$ can be viewed as a sub-Gaussian random vector (conditional on $\xi$) scaled by a constant. Therefore, we deduce that $X^T \xi \in SG(\lVert \xi \rVert_2^2)$. The last two inequalities are derived by applying Theorem \ref{norme} and Lemma \ref{control_noise}. Finally, we use the properties of the operator norm to control the term $\left( \lVert \tilde{X}_{C} (\hat{\beta}-\beta^*)\lVert \right) $:
 \begin{equation}
 \label{noise_term_2}
     \begin{aligned}
         \lVert \tilde{X}_{C} (\hat{\beta}-\beta^*)\lVert_{2}^2 &\leq  \lVert X_{C} (\hat{\beta}-\beta^*)\lVert^2\\
         &\leq \lVert X_C \lVert_{op}^2 \lVert \hat{\beta}-\beta^*\lVert^2\\
         &\leq \lambda_{\max}(XX^T)  \lVert \hat{\beta}-\beta^*\lVert^2.
     \end{aligned}
 \end{equation}
First inequality is true since: $\forall\ \beta \in \mathbf{R}^p,\ 1\leq i\leq n\ w_i(\beta)^2 \leq 1 $. Now combining the previous result with (\ref{noise term},\ref{noise_term_1},\ref{noise_term_2}) we get:  
\begin{equation}
\label{clean_final_bound}
\frac{1}{n} \sum_{i\in \text{clean}} w_i(\hat{\beta}) X_i^T (\hat{\beta}-\beta^*) \xi_i  \leq  \sqrt{2} \frac{s_{\max}(X)}{s_{\min}(X)} \left (4 \sqrt{\frac{p}{n}}+2\sqrt{\log(1/\delta)/n}\right).
\end{equation}
\paragraph{Final bound}
We establish the following intermediate result by combining (\ref{left_hand_sid_term},\ref{outlier_term_final_bound},\ref{clean_final_bound}).
    If the following Assumptions are verified (\ref{assmp1}-\ref{assmp3}), then with probability at least 1-$\delta$:
    
\begin{align*}
        \lVert \hat{\beta}-\beta^*\rVert \leq  e^{\frac{1}{4}} &\left( \sqrt{\frac{2o'}{\tau e}} \left(\frac{\max_{|O| \leq 2o'} s_{\max}(X_O)}{\min_{|S| \geq Dn} \lambda_{\min}(X_SX_S^T)}\right) \right)+ 
        \\ 
        &\left(\sqrt{2 n}\left(\frac{s_{\max}(X)}{ \min_{|S| \geq Dn} \lambda_{\min}(X_SX_S^T) s_{\min}(X) } \right) (4\sqrt{p}+2\sqrt{\log(1/\delta)}) \right) .
\end{align*}
Thanks to Lemma \ref{valeurs_singulière}, we know that $s_{\max}(X)$ and $s_{\min}(X)$ are of the order of $\sqrt{n}$, respectively, with probability at least $1-\exp(-cn)$. Furthermore, in the proof of Theorem \ref{conv}, we demonstrated that $\min_{|S| \geq \frac{3n}{4}} \lambda_{\min}(X_S X_S^T)$ is of the order of $\frac{3n}{4} - \sqrt{\frac{3np}{4}}$. Therefore, it remains to control the term $\max_{|O| \leq 2o'} s_{\max}(X_O)$. To address this, as in the proof of Theorem \ref{conv}, we apply a union bound.

\begin{align*}
    \mathbf{P}\left( \max_{|O|\leq 2o'} \frac{s_{\max}(X_O)}{|O| +C\sqrt{p}+t} \geq 1\right)&= \mathbf{P}\left( \exists O, |O|\leq 2o', \frac{s_{\max}(X_O)}{|O|+C\sqrt{p}+t} \geq 1 \right)\\
    &\leq \mathbf{P}\left(\bigcup_{i}^{2o'} |O|=i, \frac{s_{\max}(X_O)}{|O|+c\sqrt{p}+t} \geq 1 \right)\\
    &\leq \sum_{i=0 }^{2o'} \binom{n}{i} 
    \mathbf{P}\left( s_{\max}(X_O) \geq i+c\sqrt{p}+t \right)\\
    & \leq e^{-ct}\underbrace{\sum_{i=0 }^{2o'} \binom{n}{i}}_{\leq \left( \frac{en}{2o'}\right)^{2o'}}  \\
    &\leq \exp\left(-c \left( t^2-\frac{2o'}{c}\log\left(\frac{en}{2o'}\right)\right)\right).
\end{align*}
The previous result holds for all positive $t$. However, it is valid only if: $ t^2-\frac{2o'}{c}\log\left(\frac{en}{2o'}\right) >0$.  We therefore introduce the intermediate variable $u:= t^2-\frac{2o'}{c}\log\left(\frac{en}{2o'}\right)$. Furthermore, given that the vectors $(X_i)_{i=1,\dots,n}$ are 1-sub-Gaussian, the constants $C$ and $c$ can be taken as universal constants, which yields the following result: There exist two universal constants $C, c > 0$ such that with probability $1-\exp(-cu)$: 
\begin{equation}
    \max_{|O|\leq 2o'} s_{\max}\leq \sqrt{2o'}+C\sqrt{p}+\sqrt{u+\frac{2o'}{c} \log\left(\frac{en}{2o'}\right)}.
\end{equation}
Finally, with probability at least $1-\delta$ there exists a universal constant sufficiently large $C_1$ such that: \\
\begin{equation}
        \lVert \hat{\beta}-\beta^*\rVert \leq  C_1 \left( \frac{1}{\sqrt{\tau}}\frac{2o'}{ n}\sqrt{\log\left(\frac{en}{2o'}\right)} +  \sqrt{\frac{p}{n}}+\sqrt{\frac{\log(1/\delta)}{n}} \right) .
\end{equation}
\section{Proof of proposition \ref{born}}
Using Proposition \ref{convex_set} and the strict convexity of $f$ on $\mathscr{J}_c$ proved in Theorem \ref{conv}, we know that the solution of \eqref{estimator}, initialized with the LAD estimator, is unique and belongs to $\mathcal{O}_\tau$ with probability $1-\delta$. Moreover using Lemma \ref{lem:o'}, we know that $o' \leq C(o+\log(1/\delta))$ with probability at least $1-\delta$. Hence we can invoke 
the bound in Theorem \ref{deviation_bound_thm} and replace $\tau$ by its value.
\section{Proof of Theorem \ref{debiais}}
The proof of this theorem is largely analogous to that of Theorem \ref{deviation_bound_thm}. When the outliers become sufficiently large, we exploit the fact that the function:
\begin{align*}
    g: \mathbb{R}^{+}& \longrightarrow \mathbb{R}^{+} \\
    x & \longmapsto x \exp(-x)
\end{align*}
is decreasing on $[1, +\infty[$ to more effectively control the term \eqref{term_1_outliers}. Notice first that for every observation $i$ in the outliers set, for any $\hat{\beta} \in \mathscr{J}_c  $,  we have the following:
\begin{align*}
    |\theta_i+X_i^T (\hat{\beta}-\beta^*) | &\geq |\theta_i|-|X_i^T (\hat{\beta}-\beta^*)|\\
    &\geq \min_{i\in \text{outliers}} |\theta_i|- \max_{i \in [n]} \lVert X_i \lVert \lVert \hat{\beta}-\beta^* \lVert\\
    & \geq \min_{i\in \text{outliers}} |\theta_i| - \max_{i \in [n]} \lVert X_i\lVert\\
    (\text{with probability}\ 1-\delta) \
    &\geq \min_{i\in \text{outliers}} |\theta_i| - 4 \left(\sqrt{p}+\sqrt{\log(n})+ \sqrt{\log(1/\delta)} \right). 
\end{align*}
The third inequality is obtained using the definition of being an element of the set $\mathscr{J}_c$ for $c$ small enough: $\hat{\beta} \in \mathscr{J}_c \Leftrightarrow  \lVert \hat{\beta}-\beta^* \lVert \leq c \Rightarrow \lVert \hat{\beta}-\beta^* \lVert \leq 1 $. The final inequality is obtained using Theorem \ref{norme} as well as a union bound. If we suppose in addition that $\min_{i\in \text{outliers}} |\theta_i| \geq 8 \left(\sqrt{p}+\sqrt{\log(n})+ \sqrt{\log(1/\delta)} \right)$, we obtain the following bound: 
\begin{equation}
    \label{upper_bound_debias}
    |\theta_i+X_i^T (\hat{\beta}-\beta^*) | \leq \frac{\underset{i\in \text{outliers}}{\min}|\theta_i|}{2}. 
\end{equation}
Thus the bound of Theorem \ref{deviation_bound_thm} becomes with probability $1-\delta$: 

\begin{align*}
        \lVert \hat{\beta}-\beta^*\rVert \leq  C_2 \left( \frac{o'}{ n} \sqrt{\log\left(\frac{en}{2o'}\right)} 
        \min_{i\in \text{outliers}} |\theta_i| \exp\left(\frac{-\tau \underset{i\in \text{outliers}}{\min}|\theta_i|^2 }{8} \right)
        +  \sqrt{\frac{p}{n}}+\sqrt{\frac{\log(1/\delta)}{n}} \right).
\end{align*}
Using that $ \sqrt{\tau} \min_{i\in \text{outliers}} |\theta_i| \geq C \left(\sqrt{p}+\sqrt{\log(n}) \right)$, we obtain: 
\begin{equation}
\label{debias_draft_bound}
     \frac{o'}{ n} \sqrt{\log\left(\frac{en}{2o'}\right)} 
        \min_{i\in \text{outliers}} |\theta_i| \exp\left(\frac{-\tau \underset{i\in \text{outliers}}{\min}|\theta_i|^2 }{8} \right) \leq  C_3 \sqrt{\frac{p}{n}}.
\end{equation}

\section{Proof of Theorem \ref{assymp}}
One potential choice for the sequence \( u_n \)  in the definition of \( \tau_n \) is \( \log(n) \). Going back to the first-order condition of problem \eqref{estimator}, we obtain: 
\begin{align*}
    &\frac{1}{n} \sum_{i=1}^n w_i(\hat{\beta}) X_i X_i^T (\hat{\beta}-\beta^*)= \frac{1}{n} \sum_{i=1}^n w_i(\hat{\beta}) X_i \xi_i \\
    \Leftrightarrow 
    &\left( \frac{1}{n} \sum_{i=1}^n (w_i(\hat{\beta})-1) X_iX_i^T +\frac{1}{n} \sum_{i=1}^n  X_iX_i^\top \right) \sqrt{n}(\hat{\beta}-\beta^*)\\
    &=\frac{1}{\sqrt{n}} \sum_{i=1}^n  X_i \xi_i+\frac{1}{\sqrt{n}} \sum_{i=1}^n( w_i(\hat{\beta})-1) X_i \xi_i. 
\end{align*}
First, since \( w_i(\cdot) \) is bounded and \( X_i X_i^T \) is definite positive, for all  \( 0 \leq i \leq n \), and that \( \mathbf{E}\left[ X_1 X_1^T \right] = \identite_p \), we establish, thanks to the law of large numbers: 
\[
\frac{1}{n} \sum_{i=1}^n  X_i X_i^T \xrightarrow{a.s.} \identite_p.
\]
We aim to demonstrate that, for the choice of \( \tau_n \) as established in the statement of the theorem and utilizing the bound obtained in Theorem \ref{deviation_bound_thm}, the term \( \frac{1}{n} \sum_{i=1}^n (w_i(\hat{\beta})-1) X_i X_i^\top \) tends to \( 0 \). To accomplish this, we employ a Taylor expansion of \( w_i(\cdot) \) in a neighborhood of \( \beta^* \). Since \( w_i(\cdot) \) is differentiable in the vicinity of  \( \beta^* \) for every \( 0 \leq i \leq n \), we can derive the following result:
\[
w_i(\hat{\beta})-1 = w_i(\beta^*)-1  + \nabla w_i(\beta^*)^T (\hat{\beta} - \beta^*) + R_i,
\]
where \( R_i \) is the remainder term in the Taylor expansion, such that with  probability $1-1/n$, for all \( 0 \leq i \leq n \), \( R_i \leq \tau_n  \lVert \hat{\beta} - \beta^* \rVert^2  \).  Moreover, as established by Theorem \ref{deviation_bound_thm}, \( \lVert \hat{\beta} - \beta^* \rVert \) converges to zero under the specified choice of \( \tau_n \) given in the statement and taking, for instance, \( \delta = \frac{1}{n} \) in Theorem \ref{deviation_bound_thm}. 

Furthermore, for any \( 0 \leq i \leq n \), we have \( \nabla w_i(\beta^*) = -\tau_n w_i(\beta^*) X_i^T \). Recall that \( w_i(\beta^*) = \exp\left( -\frac{\tau_n}{2} \xi_i^2 \right) \). Since, by Assumption, \( (X_i)_{i \in [1,n]} \) and \( (\xi_i)_{i \in [1,n]} \) are independent, and given that the function \( \exp\left( -\frac{\tau_n}{2} x^2 \right) \) is measurable, it follows that \( (w_i(\beta^*))_{i \in [1,n]} \) and \( (X_i)_{i \in [1,n]} \) are also independent. We get with probability, at least $1-1/n$, that
$$
\underset{i}{\max} |w_i(\hat{\beta})-1 | \leq \tau_n (\log(n) + \sqrt{p}\|\hat{\beta} - \beta^*\| + \|\hat{\beta} - \beta^*\|^2).
$$
It comes out that, with probability, at least $1-1/n$,
$$
\left\|\frac{1}{n} \sum_{i=1}^n (w_i(\hat{\beta})-1) X_i X_i^\top \right\| \leq C \tau_n \log(n),
$$
since $\|\hat{\beta} - \beta^*\|\leq 1$. As a consequence, the term \( \frac{1}{n} \sum_{i=1}^n (w_i(\hat{\beta})-1) X_i X_i^\top \) tends to \( 0 \) in probability. We conclude that  
\begin{equation}\label{clt1}
    \left( \frac{1}{n} \sum_{i=1}^n (w_i(\hat{\beta})-1) X_iX_i^T +\frac{1}{n} \sum_{i=1}^n  X_iX_i^\top \right)  \xrightarrow{\mathbf{P}} \identite_p.
\end{equation}
For the right-hand term, since the variables \( (X_i \xi_i)_{i \in [1,n]} \) have zero expectation and finite covariance, we can apply the central limit theorem. Additionally, using the fact that \( \mathbf{E}[X_1] = 0 \) by Assumption, we have:
\[
\frac{1}{\sqrt{n}} \sum_{i=1}^n X_i \xi_i \xrightarrow{d} \mathcal{N}(0, \identite_p).
\] 
To control the term \( \frac{1}{\sqrt{n}} \sum_{i=1}^n( w_i(\hat{\beta})-1) X_i \xi_i\), we once again employ a Taylor expansion of \( w_i(\cdot) \) around \( \beta^* \). The remainder term are treated in a manner analogous to the one explained in the first part of the proof. We have, with probability at least $1-1/n$, that
$$
\left\|\frac{1}{\sqrt{n}} \sum_{i=1}^n( w_i(\hat{\beta})-1) X_i \xi_i\right\| \leq 
\underset{i}{\max} |w_i(\hat{\beta})-1 | \|X\xi\| \leq C\tau_n \log(n)\sqrt{n (p+\log(n)}.
$$
It comes out that \( \frac{1}{\sqrt{n}} \sum_{i=1}^n( w_i(\hat{\beta})-1) X_i \xi_i \xrightarrow{\mathbf{P}} 0\). Applying Slutsky's Lemma, we get that 
\begin{equation}\label{clt2}
    \frac{1}{\sqrt{n}} \sum_{i=1}^n  X_i \xi_i+\frac{1}{\sqrt{n}} \sum_{i=1}^n( w_i(\hat{\beta})-1) X_i \xi_i \xrightarrow{\mathbf{P}} \mathcal{N}(0, \identite_p).
\end{equation}
Finally, we conclude the proof by invoking Slutsky's Theorem one final time combining \eqref{clt1} and \eqref{clt2}.
\end{document}